\documentclass[a4paper,11pt]{article}
\usepackage[utf8]{inputenc}
\usepackage{amsmath,amssymb}
\usepackage{amsthm}
\usepackage{mathtools}
\usepackage{graphicx}
\usepackage{ wasysym }
\usepackage{enumerate}

\RequirePackage[unicode]{hyperref}
\hypersetup{
  bookmarksopen,
  bookmarksopenlevel=1,
  pdfborder={0 0 0},
  pdfdisplaydoctitle,
  pdfpagelayout={SinglePage},
  pdfstartview={FitH}
}
\usepackage[margin=3cm]{geometry}
\usepackage{dirtytalk}

\usepackage{lipsum}
\usepackage{anyfontsize}

\DeclareMathOperator{\supp}{supp}
\DeclareMathOperator{\T}{\mathbb{T}}

\newtheorem{theorem}{Theorem}[section]
\newtheorem{lemma}[theorem]{Lemma}
\newtheorem{proposition}[theorem]{Proposition}

\theoremstyle{definition}
\newtheorem{definition}[theorem]{Definition}

\theoremstyle{remark}
\newtheorem{remark}[theorem]{Remark}
\newtheorem*{remark*}{Remark}
\newtheorem*{remarks*}{Remarks}

\numberwithin{equation}{section}

\newtheorem*{claim*}{Claim} 
\newtheorem*{rem}{Remark}

\title{Aggregation equations with fractional diffusion: preventing concentration by mixing
\vspace{.4cm}}
\makeindex

\author{Katharina Hopf and José L.\;Rodrigo
\vspace{.4cm}}
\date{}
\begin{document}
\maketitle
\begin{abstract}\noindent
We investigate a class of aggregation-diffusion equations with strongly singular kernels and weak (fractional) dissipation in the presence of an incompressible flow. Without the flow the equations are supercritical in the sense that the tendency to concentrate dominates the strength of diffusion and solutions emanating from sufficiently localised initial data may explode in finite time. The main purpose of this paper is to show that under suitable spectral conditions on the flow, which guarantee good mixing properties, for any regular initial datum the solution to the corresponding advection-aggregation-diffusion equation is global if the prescribed flow is sufficiently fast. This paper can be seen as a partial extension of Kiselev and Xu (\textit{Arch.\,Rat.\,Mech.\,Anal.} 222(2), 2016), and our arguments show in particular that the suppression mechanism for the classical 2D parabolic-elliptic Keller-Segel model devised by Kiselev and Xu also applies to the fractional Keller-Segel model (where $\triangle$ is replaced by $-\Lambda^\gamma$) requiring only that $\gamma>1$. In addition, we remove the restriction to dimension $d<4$.
\end{abstract}

\section{Introduction}\label{sec:intro}
We are interested in the question of how the presence of a (prescribed, steady) incompressible flow may alter the long-time dynamics of solutions of a class of aggregation-diffusion equations with strongly singular kernels. More specifically, our starting point is the evolutionary problem
\begin{equation}\label{agg-frac-diff}
  \partial_t\rho=-\Lambda^\gamma\rho+\nabla\cdot(\rho\nabla K\ast\rho) \;\text{ in }\;(0,\infty)\times\mathbb{T}^d
\end{equation}
subject to the initial condition $\rho(0)=\rho_0$ for some sufficiently regular density $\rho_0\ge0$. 
Here  $\gamma>1$, $\mathbb{T}^d=[-\frac{1}{2},\frac{1}{2})^d$ is the periodic box, 
 and we assume that the periodic convolution kernel $K$ has the following properties:
\begin{itemize}
 \item Smoothness away from the origin.
 \item $\nabla K(x)\sim \frac{x}{|x|^{2+a}}$ near $x=0$ for some $a\ge0$. Note that this is the case if $-K\sim|x|^{-a}$ in some neighbourhood of the origin (with the understanding $K\sim\log|x|$ if $a=0$).
 For simplicity of presentation, we will assume\footnote{The behaviour of the kernel near its singularity at the origin (including its sign) determines to a large extent the interaction modelled by the nonlinear term in~\eqref{agg-frac-diff}. Our choice of the sign guarantees a predominantly attractive interaction and 
 is essential for the construction of exploding solutions.} 
 $\nabla K(x)= \frac{x}{|x|^{2+a}}$ on $B_\varepsilon(0)$ for some $0<\varepsilon\ll1$.
\end{itemize}
In order to describe our results, we first need to introduce some fundamental properties of equation~\eqref{agg-frac-diff}. 
\paragraph{Conservation of mean.}
First note that formally for any solution to equation~\eqref{agg-frac-diff} the mean value is conserved in time:
\begin{align*}
 \int_{\mathbb{T}^d}\rho(t,x)dx=\int_{\mathbb{T}^d}\rho_0(x)dx.
\end{align*}
In fact, all evolution equations which we shall consider here enjoy this property, and in this context we will abbreviate $\bar\rho=\int_{\mathbb{T}^d}\rho_0.$ In applications $\rho$ usually describes a density, and for the sake of exposition, we will henceforth assume $\rho_0\ge0$, a property, which by the maximum principle (see e.g.\,\cite{li_wellposedness_2010} for a proof in a related setting) is preserved in time for any sufficiently regular solution to~\eqref{agg-frac-diff}. 
It will, however, be obvious that (apart from the blowup proof in Appendix~\ref{app:blowup}) our results remain valid without the assumption of positivity. 

\paragraph{Scaling.}
Let us for the moment replace $\mathbb{T}^d$ by $\mathbb{R}^d$ and consider the scaling properties of the equation obtained by substituting in~\eqref{agg-frac-diff} the kernel $\nabla K$ for its homogeneous approximation near the origin, i.e.\,$\frac{x}{|x|^{2+a}}$.
This equation is invariant under the scaling 
\begin{align}\label{eq:scale}
  \rho_\lambda(t,x)=\lambda^{\gamma-2+d-a}\rho(\lambda^{\gamma}t,\lambda x),\;\;\;\lambda>0.
\end{align}
Moreover, by preservation of mean, non-negative solutions have conserved $L^1_x$-norm. 
Thus, the exponent $\gamma=\gamma_c$ which leaves the $L^1_x$-norm of the rescaled solutions $\rho_\lambda$ invariant in the sense that $\|\rho_\lambda(t,\cdot)\|_{L^1}=\|\rho(\lambda^{\gamma} t,\cdot)\|_{L^1}$ plays a distinguished role and is generally referred to as the $L^1$-\textit{critical} exponent. From~\eqref{eq:scale} we obtain 
\begin{equation*}
 \gamma_c=2+a.
\end{equation*}
In this degenerate case the conservation of mean property does not provide any control of the scaling parameter~$\lambda$ and, in principle, it would allow for self-similar blowup. 
For $\gamma<\gamma_c$ (resp.\,$\gamma>\gamma_c$) equation~\eqref{agg-frac-diff} is called $L^1$-\textit{supercritical} (resp.\;$L^1$-\textit{subcritical}).
In the case $a=0$ and $\gamma\in(1,2]$ (which implies $\gamma\le\gamma_c$) it is not difficult to produce\footnote{A proof is  given in Appendix~\ref{app:blowup}.} solutions exploding in finite time for suitably localized smooth initial data of large mass using a virial type argument similar to the strategy in~\cite[Appendix~I]{kiselev_suppression_2016}. 
 This reflects the above scaling heuristics: in the $L^1$-supercritical (and critical) regime,
diffusion is too weak to be generically able to compete with the aggregation effects induced by the quadratic drift term in~\eqref{agg-frac-diff} with velocity $\nabla K\ast\rho$. 
Conceptually, one would therefore also expect the existence of blowup solutions for more singular kernels ($a>0$) as long as diffusion is not too strong (at most critical). However, the standard moment method does not seem to be applicable directly in this case as for $a>0$ the arising \say{perturbation} terms 
\begin{align}\label{eq:bad-term}
  \int\int\frac{x-y}{|x-y|^{a+2}}\cdot\Psi(x,y)\rho(x)\rho(y)\,dydx
\end{align}
(with $\Psi$ being some smooth cut-off which in general does not vanish along the diagonal)
can no longer be controlled only in terms of the (conserved) mass $\int\rho$. The emergence of terms of the form~\eqref{eq:bad-term} in the moment method seems to be unavoidable on $\mathbb{T}^d$ even if restricting to more specific kernels $K$ -- the reason being that polynomials (such as $|x|^2$) are not periodic.

\paragraph{Background and results.}
One of our main goals (cf.\;Theorem~\ref{thm:suppression}) is to show that there exists an exponent $\gamma_0<\gamma_c$ such that (the expected) blowup can be suppressed through the action of a suitable fast flow whenever $\gamma\in(\gamma_0,\gamma_c]$. This question is motivated by the work of Kiselev and Xu~\cite{kiselev_suppression_2016} where the authors prove a similar statement for the two- and three-dimensional parabolic-elliptic Keller-Segel model (one of the fundamental models for aggregation in several biological and physical systems). 
The class of flows we focus on is a generalisation of weakly mixing flows in the ergodic sense, and a natural adaptation of the class of \textit{relaxation enhancing} flows considered in~\cite{kiselev_suppression_2016} to the case of fractional dissipation. The notion of \text{relaxation enhancing} flows was introduced in the work~\cite{constantin_diffusion_2008} by Constantin, Kiselev, Ryzhik and Zlato\v{s}, which constitutes a core reference for our approach.
Notice that for $a=d-2$ the kernel $K$ has the same singularity at the origin as the Newton kernel and, informally speaking, in this case equation~\eqref{agg-frac-diff} becomes the fractional (or classical if $\gamma=2$) parabolic-elliptic Keller-Segel model. In this sense our model is a generalisation of Keller-Segel and indeed, virtually the same analysis as in this paper can be used to give a direct derivation of the corresponding results for Keller-Segel.
For a biological motivation and more background on both (chemotactic) aggregation phenomena as well as fluid mixing and its possible regularising effects, we refer to~\cite{kiselev_suppression_2016} and references therein.
We conclude by pointing out another interesting work \cite{bedrossian_suppression_2016}, which  demonstrates that chemotactic singularity formation can also be prevented by mixing due to a fast shear flow. The underlying mechanism is, however, rather different from the one considered here and is not able to suppress more than one dimension (of the Keller-Segel model which is $L^1$-critical for $d=2$ and $L^1$-supercritical in higher dimensions). Our second main result (Theorem~\ref{thm:suppKS}) will show that the suppression mechanism by ergodic type mixing has a much weaker dimensional dependence in the sense that it applies to the Keller-Segel model in arbitrarily high dimension.

\bigskip

We finish this section by introducing two technical assumptions on the kernel $K$ needed in large parts of our analysis, commenting on local properties of solutions to~\eqref{agg-frac-diff}, fixing basic notations and indicating the organisation of the rest of this text.
\paragraph{Further assumptions on \texorpdfstring{$K$}{}.}
For fixed $\varepsilon>0$ and $p_0>1$ we note
\begin{align*}
 \int_{B_\varepsilon(0)}\frac{1}{|x|^{(1+a)p_0}}\,dx=c_d\int_0^\varepsilon r^{d-1-p_0(1+a)}dr,
\end{align*}
which shows that $\nabla K\in L^{p_0}(\mathbb{T}^d)$ if and only if 
\begin{align}\label{gradLP}
 p_0<\frac{d}{1+a}.
\end{align}
In the following we will therefore assume that the parameters $d\ge2$ (integer) and $a\ge0$ are such that $\frac{d}{1+a}>1$, so that in particular there always exists $p_0>1$ satisfying \eqref{gradLP}.

 Moreover, since we focus on $L^2$-methods in our first main result (cf.~Footnote~\ref{fn:gamma<1}), we will assume for this part that $2+a-\frac{d}{2}<2$, or equivalently, 
 \begin{align}\label{eq:L2cond}
  \frac{d}{2a}>1.
 \end{align}
 This condition ensures that the lower bound $\gamma_0=2+a-\frac{d}{2}$ on $\gamma$, which makes the $L^2$-norm (heuristically) a subcritical quantity for~\eqref{agg-frac-diff}, is less than $2$. 

\paragraph{LWP and smoothing.}\label{page:smoothing}
If $\gamma>1$, problem~\eqref{agg-frac-diff} is locally well-posed in $H^s(\mathbb{T}^d)$ for sufficiently large $s\ge s_0(d)$. 
More specifically, if\hspace{.05cm}\footnote{Notice that for $\gamma=2+a-d\left(1-\frac{1}{p}\right)$ the scaling~\eqref{eq:scale} preserves the $L^p_x$-norm in the sense that $\|\rho_\lambda(t,\cdot)\|_{L^p}=\|\rho(\lambda^\gamma t,\cdot)\|_{L^p}$ so that the required strength of diffusion for making the $L^p$ norm heuristically subcritical decreases with increasing $p$.
  Thus, one may expect to obtain improved lower bounds on $\gamma$ by working in $L^p$ spaces of higher integrability. In Theorem~\ref{thm:suppKS} we will illustrate that in some sense this is indeed the case using the example of the  standard Keller-Segel model. 
  In two spatial dimensions, for Keller-Segel type singularities ($a=d-2$)
  $L^2$ methods work for any $\gamma>1$, which is why we first focus on the case $p=2$. See also the discussion in Section~\ref{sec:suppression} (page~\pageref{eq:necessaryEst}) for difficulties arising in $L^p$.\label{fn:gamma<1}}
\begin{align}\label{eq:LpSubcriticalCond}
\gamma>\max\left\{2+a-d\left(1-\frac{1}{p}\right),1\right\},
\end{align} then local existence and uniqueness already hold in $L^p(\mathbb{T}^d)$. 
This can be shown using semigroup estimates for $-\Lambda^\gamma$ and a fixed point argument similar to~\cite{kiselev_biomixing_2012_enhancement} and~\cite{biler_blowup_2010}. 

Throughout these notes, for simplicity of exposition, we will formulate auxiliary results under the assumption of a smooth initial datum $\rho_0$ (resp.\,a smooth solution). This assumption can be removed by standard arguments exploiting the fact that, as soon as condition~\eqref{eq:LpSubcriticalCond} holds true, the smoothing effect induced by $-\Lambda^\gamma$ is strong enough to instantaneously regularise the (local) solution emanating from an $L^p$ datum. 

\paragraph{Notations.}

For smooth periodic functions $f(x)=\sum_{k\in\mathbb{Z}^d}\hat f(k) \text{e}^{2\pi \text{i}x\cdot k}$ and $\sigma \in\mathbb{R}$ we define
$$\|f\|_{\dot H^\sigma}^2=\sum_{k\in\mathbb{Z}^d\setminus\{0\}}|k|^{2\sigma}|\hat f(k)|^2$$ 
and 
$$\|f\|_{H^\sigma}^2=\sum_{k\in\mathbb{Z}^d}\langle k\rangle^{2\sigma}|\hat f(k)|^2,$$
where $\langle k\rangle=(1+|k|^2)^\frac{1}{2}$.
The space $H^\sigma(\mathbb{T}^d)$ is defined as the completion of $C^\infty(\mathbb{T}^d)$ under the norm $\|\cdot\|_{H^\sigma}$.
We next define  the fractional derivative $\Lambda^\sigma$  via 
\begin{align}\label{def:fractional}
  \Lambda^\sigma f(x)=\sum_{k\in\mathbb{Z}^d\setminus\{0\}}|k|^\sigma\hat f(k)\text{e}^{2\pi\text{i}k\cdot x}.
\end{align}
For sufficiently regular periodic functions $f,g$ the following identities are immediate
\begin{align*}
  \|f\|_{\dot H^\sigma}&=\|\Lambda^\sigma f\|_{L^2},
  \\\Lambda^\sigma(f\ast g)&=f\ast\Lambda^\sigma g.
\end{align*}
Moreover,
\begin{align*}
  \int_{\mathbb{T}^d}f\Lambda^\sigma g&= \int_{\mathbb{T}^d}(\Lambda^\sigma f)\,g,\\
  \Lambda^{\sigma_1}\Lambda^{\sigma_2}f&=\Lambda^{\sigma_1+\sigma_2}f.
\end{align*}

Constants $C$ or $C(\dots)$ may change from line to line and unless explicitly  indicated otherwise they are continuous and non-decreasing functions of their (non-negative) arguments. Their possible dependence on the parameters $\gamma,a$ and $d$ will usually not be indicated explicitly.
For quantities $A,B\ge0$ the notation $A\lesssim B$ means that there exists a constant $0<C<\infty$ (which may depend on fixed parameters) such that $A\le CB$. Furthermore, $A\sim B$ stands for $A\lesssim B$ and $B\lesssim A$. If it is appropriate to indicate the dependence of the hidden constant in \say{ $\lesssim$ } on certain parameters $p_1,\dots$, this will be done through\, $\lesssim_{p_1,\dots}$.

\paragraph{Outline.} This text is structured as follows. In the next section we recall several well-known estimates needed for the subsequent analysis. 
Section~\ref{sec:apriori} is devoted to the derivation of a priori estimates required for our $L^2$-based suppression result.
In Section~\ref{sec:suppression} we first introduce further concepts in order to determine the flows leading to the specific prevention of concentration mechanism which we here focus on. Then we turn to the proof of our main results.  
\\In Appendix~\ref{app:blowup} the existence of exploding solutions to equation~\eqref{agg-frac-diff} is proved in the case $a=0$, $\gamma\in(1,2]$. The appendix to this text further contains two extensions of results in the literature which we require for our main argument in Section~\ref{sec:suppression} (see Appendices~\ref{app:tranport-diffusion} and~\ref{app:transport}). Finally, in Appendix~\ref{app:RE-examples} we construct examples of incompressible flows, which provide a justification for our Definition~\ref{def:gammaRE} of $\gamma$-relaxation enhancing flows.

\section{Auxiliary tools}\label{sec:aux}
 
Here we collect some standard inequalities, which will be used throughout the text.

\begin{lemma}[Interpolation]\label{l:interpol}Let $\sigma,\mu>0$. Then for all $f\in C^\infty(\mathbb{T}^d)$
 \begin{equation}
  \|f\|_{\dot H^\sigma}\lesssim\|f\|_{L^2}^{1-b}\|f\|_{\dot H^{\sigma+\mu}}^{b},
 \end{equation} 
 where $b=\frac{\sigma}{\sigma+\mu}$.
\end{lemma}
\begin{proof}
 We compute using Plancherel's identity and H\"older inequality with $p=\frac{1}{1-b}$
 \begin{align*}
  \|f\|^2_{\dot H^\sigma}&=\int|\Lambda^\sigma f|^2\,dx\approx \sum_k|k|^{2\sigma}|\hat f(k)|^2=\sum_k|\hat f(k)|^{2(1-b)}|k|^{2\sigma}|\hat f(k)|^{2b}\\
  &\le\left(\sum_k|\hat f(k)|^2\right)^{({1-b})}\left(\sum_k|k|^{2(\sigma+\mu)}|\hat f(k)|^{2}\right)^{b},
 \end{align*}
 where in the last step we used $\frac{\sigma}{b}=\sigma+\mu$.
\end{proof}
The following result is an immediate consequence of Plancherel's identity and Cauchy-Schwarz.
\begin{lemma}[Duality]\label{l:duality}
 Let $f,g\in C^\infty(\mathbb{T}^d)$ satisfy $\hat f(0)\hat g(0)=0.$ Then for $\sigma\in\mathbb{R}$
 \begin{align*}
  \int_{\mathbb{T}^d}f(x)g(x)\,dx\le\|f\|_{\dot H^\sigma}\|g\|_{\dot H^{-\sigma}}.
 \end{align*}
\end{lemma}
In our analysis we will frequently use the following product rule estimate (also known as Kato-Ponce inequality) combined with the subsequently stated Sobolev embedding for fractional derivatives.
\begin{lemma}[Fractional product rule estimate]\label{l:product} Let $\sigma\ge0$ be given.
Then for all $p_i,q_i\in(2,\infty)$ with $\frac{1}{2}=\frac{1}{p_i}+\frac{1}{q_i}$, $i=1,2$ the bound
\begin{equation}
 \|\Lambda^\sigma (fg)\|_{L^2}\lesssim\|\Lambda^\sigma f\|_{L^{p_1}}\|g\|_{L^{q_1}}+\|f\|_{L^{p_2}}\|\Lambda^\sigma  g\|_{L^{q_2}}
\end{equation} holds true.
\end{lemma}
\begin{proof}
For the whole space this is a special case of e.g.~\cite{grafakos_kato-ponce_2014}. In the case of the torus, we refer to~\cite{constantin_unique_2014} and references therein.
\end{proof}

\begin{lemma}[Homogeneous Sobolev embedding]\label{l:Sobolev}
 Assume $0<\frac{\sigma}{d}<\frac{1}{p}<1$ and define $q\in(p,\infty)$ via
 \begin{equation*}
 \frac{\sigma}{d}=\frac{1}{p}-\frac{1}{q}.
 \end{equation*}
 Then for all $f\in C^\infty(\mathbb{T}^d)$ with zero mean
\begin{equation}
 \|f\|_{L^q(\mathbb{T}^d)}\lesssim\|\Lambda^\sigma f\|_{L^p(\mathbb{T}^d)}.
\end{equation} 
\end{lemma}
\begin{proof} See~\cite{benyi_sobolev_2013} for a direct Fourier analytic proof on the torus.
\end{proof}

\section{\texorpdfstring{$L^2$}{} a priori estimates}\label{sec:apriori}
In this section we will establish a priori estimates for the evolution equation
\begin{subequations}\label{eq}
\begin{align}
  \partial_t\rho+u\cdot\nabla\rho&=-\Lambda^\gamma\rho+\nabla\cdot(\rho\nabla K\ast\rho) \;\text{ in }\;(0,\infty)\times\mathbb{T}^d,\\
  \rho(0)&=\rho_0,
\end{align}
\end{subequations} 
where $u=u(x)$ is a given smooth divergence-free vector field and $\rho_0$ a non-negative initial datum. Clearly, the conservation of mean property, preservation of positivity, LWP and the smoothing effects for the local solution mentioned in the introduction remain valid for problem~\eqref{eq}. Strictly speaking, a thorough elaboration of the local well-posedness theory in $L^2$ and the smoothing properties would imply some (albeit possibly weaker) version of the results derived in this section. 
We nevertheless include this part, mainly because the derived lemmata will be used explicitly in and will facilitate the presentation of the proof of our first \say{blowup prevention theorem} (Theorem~\ref{thm:suppression}). 

To simplify the exposition, we will prove the following results only in the (more interesting) cases $\gamma\le2$ and $2+a-\frac{d}{2}\ge1$. It is not difficult to see that they remain valid if $\gamma>2$ or $2+a-\frac{d}{2}<1$. 

\subsection{A blowup criterion}\label{sec:criterion}
Here we illustrate by a formal derivation that a form of the standard blowup/continuation criteria for several classical aggregation equations (including the Keller-Segel model) is also valid for our problem.

\begin{theorem}[$L^2$-control suffices]\label{thm:criterion} Assume  $\gamma>\max\{2+a-\frac{d}{2},1\}$ and let\footnote{Recall that thanks to the assumed lower bound on $\gamma$, by the smoothing properties of~\eqref{eq}, the assumption of smooth initial data can be removed, and the statement, mutatis mutandis, is valid for $L^2$ data.} 
$\rho_0\in C^\infty(\mathbb{T}^d)$. Then the following criterion holds: 
either the local solution $\rho$ to~\eqref{eq} extends to a global smooth solution or there exists $T^*\in(0,\infty)$ and $1\le r<\infty$ such that 
\begin{equation*}
  \int_0^t \|\rho(\tau)-\bar\rho\|_{L^2}^r\,d\tau\to\infty\text{ as }t\nearrow T^*.
\end{equation*} 
\end{theorem}
\begin{proof}[Proof of Theorem~\ref{thm:criterion} (for $\gamma\le2,\;\;2+a-\frac{d}{2}\ge1$).] 
It suffices to derive a priori bounds on higher order derivatives in terms of $L^2$, the rest of the argument then follows as in~\cite[Appendix I]{kiselev_biomixing_2012_enhancement}. Let $s\ge s_0(d)$ be a sufficiently large integer.  Then we estimate as in the proof of~\cite[Theorem 2.1]{kiselev_suppression_2016} 
\begin{align}\label{s-test}
 \frac{1}{2}\frac{d}{dt}\|\rho\|_{\dot H^s}^2\le -\|\rho\|_{\dot H^{s+\frac{\gamma}{2}}}^2+C\|u\|_{C^s}\|\rho\|_{\dot H^s}^2+\left|\int\nabla\cdot(\rho\nabla K\ast\rho)(-\triangle)^s\rho\right|.
\end{align}
The last term on the RHS is estimated using Lemmata~\ref{l:duality} and~\ref{l:product}
\begin{align}\label{eq:l-est}
 &\left|\int \Lambda^s(\rho\nabla K\ast\rho)\cdot\nabla\Lambda^{s}\rho\,dx\right| \nonumber
 \lesssim\|\Lambda^s(\rho\nabla K\ast\rho)\|_{\dot H^{1-\frac{\gamma}{2}}}\|\nabla\Lambda^{s}\rho\|_{\dot H^{-1+\frac{\gamma}{2}}}
 \\&\hspace{2cm}\lesssim \left(\|\Lambda^{s+1-\frac{\gamma}{2}}\rho\|_{L^{p_1}}\|\nabla K\ast\rho\|_{L^{q_1}}+\|\rho\|_{L^{p_2}}\|\nabla K\ast\Lambda^{s+1-\frac{\gamma}{2}}\rho\|_{L^{q_2}}\right)\|\rho\|_{\dot H^{s+\frac{\gamma}{2}}}.
\end{align}
This is valid for $p_i,q_i\in(2,\infty)$ whenever $\frac{1}{p_i}+\frac{1}{q_i}=\frac{1}{2}$ for $i=1,2$.
In the following we estimate the terms on the RHS of~\eqref{eq:l-est}.
We first choose $p_1=2+\varepsilon$ for $\varepsilon>0$ sufficiently small such that for $\sigma_1=\left(\frac{1}{2}-\frac{1}{p_1}\right)d$ we have $b_1:=\frac{\sigma_1+s+1-\frac{\gamma}{2}}{s+\frac{\gamma}{2}}<1$. This is possible since $\gamma>1$. Thus, using Lemmata \ref{l:Sobolev} and \ref{l:interpol}, we find  
\begin{align*}
  \|\Lambda^{s+1-\frac{\gamma}{2}}\rho\|_{L^{p_1}}\le C\|\rho\|_{\dot H^{\sigma_1+s+1-\frac{\gamma}{2}}}\le C\|\rho-\bar\rho\|_{L^2}^{1-b_1}\|\rho\|_{\dot H^{s+\frac{\gamma}{2}}}^{b_1}.
\end{align*} 
Next, we apply Young's convolution inequality with suitable exponents $p_0,q_3\in(1,\infty)$ satisfying $1+\frac{1}{q_1}=\frac{1}{p_0}+\frac{1}{q_3}$. More precisely, we choose $p_0=\frac{d}{1+a}(1-\delta)$ for $\delta>0$ sufficiently small. 
Note that if $2+a-\frac{d}{2}\ge1$, then $\frac{d}{1+a}\le2$ so that for $\varepsilon>0$ sufficiently small $q_3\ge2$. And clearly, for $s\ge s_0(d)$ sufficiently large we have
 $b_2:=\frac{(\frac{1}{2}-\frac{1}{q_3})d}{s+\frac{\gamma}{2}}<1.$ 
 Thus,
 \begin{align*}
  \|\nabla K\ast\rho\|_{L^{q_1}}= \|\nabla K\ast(\rho-\bar\rho)\|_{L^{q_1}}&\le \|\nabla K\|_{L^{p_0}}\|\rho-\bar\rho\|_{L^{q_3}}
  \\&\le C\|\nabla K\|_{L^{p_0}}\|\rho-\bar\rho\|_{L^2}^{1-b_2}\|\rho\|_{\dot H^{s+\frac{\gamma}{2}}}^{b_2}.
\end{align*} 
We note that 
\begin{align}\label{eq:b1b2sum}
 b_1+b_2&=\frac{\left(\frac{1}{2}-\frac{1}{p_1}\right)d+s+1-\frac{\gamma}{2}+(\frac{1}{2}-\frac{1}{q_3})d}{s+\frac{\gamma}{2}}\nonumber
 \\&=\frac{\left(\frac{1}{p_0}-\frac{1}{2}\right)d+s+1-\frac{\gamma}{2}}{s+\frac{\gamma}{2}}\nonumber
 \\&=\frac{s+\frac{\gamma}{2}-\gamma-\frac{d}{2}+\frac{d}{p_0}+1}{s+\frac{\gamma}{2}}.
\end{align}
Since $\gamma>2+a-\frac{d}{2}$, the term $-\gamma-\frac{d}{2}+1+\frac{d}{p_0}$ is strictly negative if $\delta>0$ is chosen sufficiently small. Then the strict inequality $b_1+b_2<1$ holds. \\
The terms $\|\rho\|_{L^{p_2}}$ and $\|\nabla K\ast\Lambda^{s+1-\frac{\gamma}{2}}\rho\|_{L^{q_2}}$ are treated similarly and yield bounds with only minor differences (see the proof of Lemma~\ref{local-control}).

Inserting the derived bounds into \eqref{s-test}, applying Young's inequality twice -- once with the exponent $\tilde p=\frac{2}{b_1+b_2+1}>1$ applied to the factor involving the highest power
of $\|\rho\|_{\dot H^{s+\frac{\gamma}{2}}}$ -- we obtain, after absorption, a bound of the form 
\begin{align*}
 \frac{1}{2}\frac{d}{dt}\|\rho\|_{\dot H^s}^2\le -\frac{1}{2}\|\rho\|_{\dot H^{s+\frac{\gamma}{2}}}^2+C\|u\|_{C^s}\|\rho\|_{\dot H^s}^2+C\|\rho-\bar\rho\|_{L^2}^r+C(\bar\rho)
\end{align*}
for some possibly large $r\in(1,\infty)$, $r=r(a,d,\gamma,s)$. From this estimate the conclusion can easily be deduced. 
\end{proof}

\subsection{Local control}\label{sec:local-control}
We now prove that solutions are locally controlled in $L^2(\mathbb{T}^d)$ for some time which only depends on the $L^2$-distance to the mean, the mean value and model parameters.
\begin{lemma}[Local $L^2$-control]\label{local-control}
  Suppose  $\gamma>\max\{2+a-\frac{d}{2},1\}$ and let $\rho\ge0$ be a smooth (local) solution to \eqref{eq}. Assume that 
  $\|\rho(t_0)-\bar\rho\|_{L^2}=B>0$ for some $t_0\ge0.$ Then 
  \begin{equation*}
   \|\rho(t_0+\tau)-\bar\rho\|_{L^2}\le 2B\text{ for all }0\le\tau\le\tau_0,
  \end{equation*}
  where  
  \begin{align}\label{eq:timespan}
   \tau_0=C_1(\|\nabla K\|_{L^{p_0}})^{-1}\min\left\{B^{-r_1},\bar\rho^{-r_2}\right\}>0
  \end{align} 
 for some\footnote{Recall that hypothesis~\eqref{gradLP} ensures $1<\frac{d}{a+1}$.} sufficiently large $1<p_0<\frac{d}{1+a}$, a non-decreasing function $C_1(\dots)>0$ and positive (possibly large) constants $r_i>0$ which only depend on $\gamma, d, a$ and the choice of $p_0$.
\end{lemma}
\begin{proof}[Proof of Lemma~\ref{local-control} (for $\gamma\le2$,  $2+a-\frac{d}{2}\ge1$)] By multiplying~\eqref{eq} with $\rho-\bar\rho$ and integrating in space, we obtain
 \begin{align}\label{eq:L2esti}
   \frac{1}{2}\frac{d}{dt}\|\rho-\bar\rho\|_{L^2}^2&=-\|\rho\|_{\dot H^\frac{\gamma}{2}}^2-\int\rho\nabla K\ast\rho\cdot\nabla(\rho-\bar\rho)dx\nonumber
   \\&\le-\|\rho\|_{\dot H^\frac{\gamma}{2}}^2+\|\Lambda^{1-\frac{\gamma}{2}}(\rho\nabla K\ast\rho)\|_{L^2}\|\rho\|_{\dot H^{\frac{\gamma}{2}}}.
 \end{align}
Here we used the incompressibility of the flow.
By Lemma~\ref{l:product}, for $p_i,q_i\in(2,\infty)$ with $ {p_i}^{-1}+q_i^{-1}=2^{-1}$,  $i=1,2$, we have
\begin{align}\label{eq:for-l}
 \|\Lambda^{1-\frac{\gamma}{2}}(\rho\nabla K\ast\rho)\|_{L^2}\le C\left(\|\Lambda^{1-\frac{\gamma}{2}}\rho\|_{L^{p_1}}\|\nabla K\ast(\rho-\bar\rho)\|_{L^{q_1}}+\|\rho\|_{L^{p_2}}\|\nabla K\ast\Lambda^{1-\frac{\gamma}{2}}\rho\|_{L^{q_2}}\right),
\end{align} 
which means that
the last term on the RHS of~\eqref{eq:L2esti} can be bounded from above by 
\begin{align}\label{eq:0815}
 C\left(\|\Lambda^{1-\frac{\gamma}{2}}\rho\|_{L^{p_1}}\|\nabla K\ast(\rho-\bar\rho)\|_{L^{q_1}}+\|\rho\|_{L^{p_2}}\|\nabla K\ast\Lambda^{1-\frac{\gamma}{2}}\rho\|_{L^{q_2}}\right)\|\rho\|_{\dot H^{\frac{\gamma}{2}}}.
\end{align} 

We now claim that thanks to Young's convolution inequality and Gagliardo-Nirenberg-Sobolev  estimates (see Lemma~\ref{l:Sobolev} and~\ref{l:interpol}), term~\eqref{eq:0815} is controlled by
\begin{align}\label{eq:i1i2}
 C_\bigstar\|\nabla K\|_{L^{p_0}}\|\rho\|_{\dot H^{\frac{\gamma}{2}}} (I_1+I_2),
\end{align} 
where $C_\bigstar$ is a fixed positive constant (depending only on $\gamma, a$ and $d$) and 
\begin{align}
I_1&=\|\rho-\bar\rho\|_{L^2}^{2-(b_1+b_2)}\|\rho\|_{\dot H^\frac{\gamma}{2}}^{b_1+b_2},
 \\I_2&=(\bar\rho+\|\rho-\bar\rho\|_{L^2}^{1-b_3}\|\rho\|_{\dot H^\frac{\gamma}{2}}^{b_3})\|\rho-\bar\rho\|_{L^2}^{1-b_4}\|\rho\|_{\dot H^\frac{\gamma}{2}}^{b_4}.
\end{align}
Here $b_1,b_2\in[0,1)$ are obtained as in the proof of Theorem~\ref{thm:criterion} and satisfy  $b_1+b_2<1$ (we choose again $p_0=\frac{d}{a+1}(1-\delta)$ with $\delta=\delta(a,d,\gamma)>0$ (at least) as small as in Theorem~\ref{thm:criterion}). The value of $b_1+b_2$ is precisely given by setting $s=0$ in~\eqref{eq:b1b2sum}, i.e.
\begin{align}\label{eq:b12}
 b_1+b_2=\frac{-\frac{d}{2}+\frac{d}{p_0}+1}{\frac{\gamma}{2}}-1.
\end{align}
To see how the expression for $I_2$ and the exponents $b_3, b_4\in[0,1)$ arise, we proceed similarly to the proof of Theorem~\ref{thm:criterion}: since $2+a-\frac{d}{2}\ge1$ (which implies $p_0<\frac{d}{a+1}\le2$), we can choose $p_1>2$ sufficiently close to $2$ such that $q_4$ defined via 
\begin{align*}
 1+\frac{1}{q_2}&=\frac{1}{p_0}+\frac{1}{q_4}
\end{align*} satisfies $q_4\ge2.$
We now apply Young's convolution inequality to the second convolution term in~\eqref{eq:0815} 
estimating $\nabla K$ in $L^{p_0}$ and use in a subsequent step Lemma~\ref{l:Sobolev} (twice) for the arising  $\rho$-terms $\|\rho\|_{L^{p_2}}$ and $\|\Lambda^{1-\frac{\gamma}{2}}\rho\|_{L^{q_4}}$ 
 with 
\begin{align*} 
 \sigma_3&=\left(\frac{1}{2}-\frac{1}{p_2}\right)d,\\
 \sigma_4&=\left(\frac{1}{2}-\frac{1}{q_4}\right)d
\end{align*} and then Lemma~\ref{l:interpol} (twice) with
\begin{align*} 
 b_3&=\frac{\sigma_3}{\gamma/2},\\
 b_4&=\frac{\sigma_4+1-\gamma/2}{\gamma/2}
\end{align*}
to obtain the $I_2$-part of~\eqref{eq:i1i2}. Notice that
\begin{align}\label{eq:b3b4sum}
 b_3+b_4&=\frac{\sigma_3+\sigma_4+1-\gamma/2}{\gamma/2}\nonumber
 \\&=\frac{(1-(p_2^{-1}+q_4^{-1}))d+1}{\gamma/2}-1\nonumber
\\&=\frac{(p_0^{-1}-2^{-1})d+1}{\gamma/2}-1
\end{align}
and that the assumption $\gamma>2+a-\frac{d}{2}$ implies that for $p_0<\frac{d}{1+a}$ sufficiently large the strict bound 
$\frac{(p_0^{-1}-2^{-1})d+1}{\gamma/2}-1<1$ holds true. Hence 
\begin{align*}
 b_3+b_4<1.
\end{align*} (Since $b_i\ge0$, this justifies in particular the application of Lemma~\ref{l:interpol} above.) Note that comparison of~\eqref{eq:b12} with~\eqref{eq:b3b4sum} shows $b_1+b_2=b_3+b_4$.

Abbreviating $b:=b_3+b_4+1<2$, we thus obtain the bound
\begin{align}\label{eq:used-l}
  \frac{1}{2}\frac{d}{dt}\|\rho-\bar\rho\|_{L^2}^2\le-\|\rho\|_{\dot H^\frac{\gamma}{2}}^2+C_\bigstar\|\nabla K\|_{L^{p_0}}\left(\|\rho-\bar\rho\|_{L^2}^{3-b}\|\rho\|^{b}_{\dot H^\frac{\gamma}{2}}+\bar\rho\|\rho-\bar\rho\|_{L^2}^{1-b_4}\|\rho\|_{\dot H^\frac{\gamma}{2}}^{1+b_4}\right).
\end{align}
For later use, we remark that from~\eqref{eq:for-l} and the  subsequent estimates up to~\eqref{eq:used-l}, we immediately deduce
\begin{align}\label{eq:rkr-l}
 \|\Lambda^{1-\frac{\gamma}{2}}(\rho\nabla K\ast\rho)\|_{L^2}\le C_\bigstar
 \|\nabla K\|_{L^{p_0}}\left(\|\rho-\bar\rho\|_{L^2}^{3-b}\|\rho\|^{b-1}_{\dot H^\frac{\gamma}{2}}+\bar\rho\|\rho-\bar\rho\|_{L^2}^{1-b_4}\|\rho\|_{\dot H^\frac{\gamma}{2}}^{b_4}\right).
\end{align}

We now define
\begin{align*}
 c_1&=\left(1-\frac{b}{2}\right)^{-1}(3-b)
 =2\left(1+\frac{1}{2-b}\right)
\end{align*}
and note that
\begin{align*}
 \left(1-\frac{1+b_4}{2}\right)^{-1}(1-b_4)=2.
\end{align*}
Applying a standard absorption argument to~\eqref{eq:used-l}, we then find 
\begin{align}\label{eq:usedLater}
  \frac{d}{dt}\|\rho-\bar\rho\|_{L^2}^2\le-\|\rho\|_{\dot H^\frac{\gamma}{2}}^2+C_\clubsuit (\|\nabla K\|_{L^{p_0}})
  \left(\|\rho-\bar\rho\|_{L^2}^{c_1}+\bar\rho^{\frac{2}{1-b_4}}\|\rho-\bar\rho\|_{L^2}^{2}\right).
\end{align}
Once more for later use, we note that Young's multiplication inequality applied to the RHS of~\eqref{eq:rkr-l} yields 
\begin{align}\label{eq:forCaseII}
  \|\Lambda^{1-\frac{\gamma}{2}}(\rho\nabla K\ast\rho)\|_{L^2}^2\le \frac{1}{2}\|\rho\|_{\dot H^\frac{\gamma}{2}}^2+C_\clubsuit(\|\nabla K\|_{L^{p_0}})\left(\|\rho-\bar\rho\|_{L^2}^{c_1}+\bar\rho^{\frac{2}{1-b_4}}\|\rho-\bar\rho\|_{L^2}^{2}\right),
\end{align}
with the same constants $c_1$ and $C_\clubsuit$ as in~\eqref{eq:usedLater}.

Now note that $c_1\ge2$ and that, by~\eqref{eq:usedLater}, the function $f(t)=\|\rho(t)-\bar\rho\|_{L^2}^2$ satisfies 
\begin{align*}
 f'\le C_0f^{c_1/2}+C_0\bar\rho^{\frac{2}{1-b_4}} f, \hspace{1cm} f(t_0)=B^2
\end{align*} where $C_0=C_\clubsuit(\|\nabla K\|_{L^{p_0}})$.
Comparison with the explicit solution $\tilde f$ to
\begin{align*}
 \tilde f'= C_0\tilde f^{c_1/2}+C_0\bar\rho^{\frac{2}{1-b_4}} \tilde f, \hspace{1cm} \tilde f(t_0)=B^2,
\end{align*} 
which is given by
\begin{align*}
 \tilde f(t_0+t)=R^{\frac{1}{q}}\exp(C_0Rt)B^2\left(R-B^{2q}\left[\exp(C_0Rqt)-1\right]\right)^{-\frac{1}{q}}
\end{align*}
with $q=\frac{c_1-2}{2}$ and $R=\bar\rho^{\frac{2}{1-b_4}}$,
shows that
\begin{align}
 f(t_0+\tau)\le 4B^2, \hspace{.3cm}\text{whenever }0\le\tau\le \tau_0:=\delta_0 C_0^{-1}\min\left\{\frac{2}{c_1-2}B^{-(c_1-2)},\bar\rho^{-\frac{2}{1-b_4}}\right\}.   
\end{align}
Here $\delta_0>0$ is a universal constant.
\end{proof}

\section{Prevention of blowup}\label{sec:suppression}

In the following we assume that our vector field $u$ is relaxation enhancing in the sense analogous to~\cite[Definition~5.1]{kiselev_suppression_2016}, but adapted to the fractional diffusion~$-\Lambda^\gamma$. In order to give a precise definition, let us recall that a divergence-free Lipschitz vector field $u$ on $\mathbb{T}^d$ gives rise to a flow map $\Phi:\mathbb{R}\times\mathbb{T}^d\to\mathbb{T}^d,\;\;(t,x)\mapsto \Phi_t(x)$ via
\begin{align*}
 \frac{d}{dt}\Phi_t(x)&=u(\Phi_t(x)),\\
 \Phi_0&=\text{Id}_{\mathbb{T}^d},
\end{align*} where the transformations $\Phi_t$ are measure-preserving Lipschitz diffeomorphisms. Thus we obtain a one-parameter group of unitary operators $U^tf(x)=f(\Phi_t^{-1}(x))$ on $L^2(\mathbb{T}^d)$.
\begin{definition}\label{def:gammaRE} Let $\gamma\ge1$. We call a divergence-free Lipschitz vector field $u=u(x)$ $\gamma$-\textit{relaxation enhancing} ($\gamma$-RE) if the corresponding unitary operator $U^1$ does not have any non-constant eigenfunctions in $ H^\frac{\gamma}{2}(\mathbb{T}^d)$.	
\end{definition}
\begin{remarks*}$ $\vspace{-.2cm}
\begin{enumerate}[(i)]
 \item  Theorem~\ref{thm:relaxation} and Remark~\ref{rem:gamma<1} below will explain the term \say{relaxation enhancing}. 
 \item 
 The notion \say{relaxation enhancing} was first introduced in~\cite{constantin_diffusion_2008} in a somewhat more general context. The notion used in~\cite{kiselev_suppression_2016} corresponds in our definition to $2$-RE. Any flow which is weakly mixing in the ergodic sense (so that $U^1$ does not have any non-constant eigenfunctions in $L^2$) is also $\gamma$-RE for any  $\gamma$ as above. The existence of weakly mixing flows on $\mathbb{T}^d$ for any $d\ge2$ is classical and can be shown by considering suitable time changes of appropriate irrational translations on $\mathbb{T}^d$ (see~\cite[Section~6]{constantin_diffusion_2008} and references therein). A concrete example for a $2$-RE flow which is not weakly mixing can also be found in~\cite[Section~6]{constantin_diffusion_2008}.
 \item In Appendix~\ref{app:RE-examples} we show that 
 for any given $1\le \gamma_1<\gamma_2$ there exists a smooth, incompressible flow on $\mathbb{T}^2$ which is $\gamma_2$-RE but not $\gamma_1$-RE.
\end{enumerate}
\end{remarks*}
\noindent We now consider for a parameter $A\gg1$ the initial value problem 
\begin{subequations}
\label{eq:adv-agg-frac-diff}
\begin{align}
  \partial_t\rho^A+Au\cdot\nabla\rho^A&=-\Lambda^\gamma\rho^A+\nabla\cdot(\rho^A\nabla K\ast\rho^A) \;\;\text{ in }\;(0,\infty)\times\mathbb{T}^d,
  \\\rho^A(0)&=\rho_0,
\end{align}
\end{subequations}
where the kernel $K$ satisfies the conditions described in the introduction (Section~\ref{sec:intro}) and $d\ge2.$
The crucial ingredient in the proof of our first main theorem (Theorem~\ref{thm:suppression}) is the following result (cf. \cite{constantin_diffusion_2008}):
\begin{theorem}[Enhanced relaxation]\label{thm:relaxation}Let $\gamma\ge1$  and let the divergence-free smooth vector field $u$ be $\gamma$-relaxation enhancing. Then for every $\tau>0$, $\varepsilon>0$ there exists a positive constant $A_0=A_0(\tau,\varepsilon)$ such that for any $A\ge A_0$ and for any $\mu_0\in L^2(\mathbb{T}^d)$ with $\int_{\mathbb{T}^d}\mu_0=0$ the solution $\mu^A$ to 
\begin{align}\label{eq:linadvdiff}
 \partial_t\mu^A+Au\cdot\nabla\mu^A&=-\Lambda^\gamma\mu^A\;\;\text{ in }\;(0,\infty)\times\mathbb{T}^d,\\
 \mu^A(0)&=\mu_0\nonumber
\end{align}
satisfies $\|\mu^A(t)\|_{L^2}\le\varepsilon\|\mu_0\|_{L^2}$ for all $t\ge\tau$.
\end{theorem}
\begin{remark}\label{rem:gamma<1}$ $\vspace{-.2cm}
\begin{enumerate}[(i)]
 \item 
 As in the proof of the first part of Theorem~1.4 in~\cite{constantin_diffusion_2008}
 one can show that the $\gamma$-RE property of $u$ is also necessary 
 for the statement in Theorem~\ref{thm:relaxation} to hold. 
\item  If restricting to initial data in $H^\frac{1}{2}$ (instead of general $L^2$ data), one is still able to obtain enhanced relaxation for $\gamma\in(0,1)$ if the unitary evolution (cf.\;$U^1$ in Definition~\ref{def:gammaRE}) does not have any non-constant eigenfunctions in $H^\frac{\gamma}{2}$.
 \item Theorem~\ref{thm:relaxation} (at least with $\gamma=2$) remains true when $L^2$ is replaced by $L^p$ for any $p\in[1,\infty]$, see~\cite[Theorem~5.5]{constantin_diffusion_2008}.\label{rem:Lp-relax}
\end{enumerate}
\end{remark}
\noindent In the case $\gamma\ge2$ Theorem~\ref{thm:relaxation} is a consequence of the abstract criterion in~\cite{constantin_diffusion_2008} (combined with Proposition~\ref{prop:bdd-evol-1}). 
We will sketch the extension to arbitrary $\gamma\ge1$ in Appendix~\ref{app:tranport-diffusion}.
In any case, an important ingredient in the proof is the boundedness of the linear transport evolution in $ H^\frac{\gamma}{2}$ for sufficiently regular vector fields:
\begin{proposition}[Estimate for transport equation] \label{prop:bdd-evol-1}
Let $v=v(x)$ be a smooth divergence-free\footnote{The assumption $\nabla\cdot v=0$ is not necessary for the boundedness of the evolution~\eqref{eq:lin-tp} with respect to$\|\cdot\|_{\dot H^\frac{\gamma}{2}}$, see~\cite{bahouri_fourier_2011}.} vector field and assume $\gamma>0$. Then any sufficiently regular solution $\eta$ to 
\begin{subequations}\label{eq:lin-tp}
  \begin{align}
    \partial_t\eta+v\cdot\nabla\eta&=0\;\;\text{ in }\;(0,\infty)\times\mathbb{T}^d,
 \\\eta(0)&=\eta_0
\end{align} 
\end{subequations}
satisfies the bound
\begin{equation}\label{eq:transport-bound}
 \|\eta(t)\|_{\dot H^\frac{\gamma}{2}(\mathbb{T}^d)}\lesssim \exp(C(v)t)\|\eta_0\|_{\dot H^\frac{\gamma}{2}(\mathbb{T}^d)},
\end{equation} 
where $ C(v)\lesssim_{\gamma,d} \|\Lambda^{\gamma+\frac{d}{2}+1}v\|_{L^2}.$
\end{proposition}
\begin{remark*}
 A version of this result for Besov spaces on $\mathbb{R}^d$ can be found in~\cite{bahouri_fourier_2011}. In Appendix~\ref{app:transport} we will provide a proof of the above result, which is suboptimal in terms of the regularity required for the vector field $v$.
\end{remark*}

We are now in a position to turn to our first main result. From now on we let $p_0=p_0(\gamma,a,d)\in(1,\frac{d}{a+1})$ be an exponent for which both Theorem~\ref{thm:criterion} and Lemma~\ref{local-control} are valid. Also recall that by assumption~\eqref{gradLP} we have $\|\nabla K\|_{L^{p_0}}<\infty$.
For simplicity, any dependence of constants on $\gamma,a$ and $d$ will (as before) be omitted.
\begin{theorem}[Prevention of blowup for model with fractional dissipation]\label{thm:suppression}
Let $\gamma>\max\{2+a-\frac{d}{2},1\}.$  Suppose that the divergence-free smooth vector field $u(x)$ is $\gamma$-relaxation enhancing. Then for any $\rho_0\in L^2(\mathbb{T}^d)$ there exists $A_0(\|\rho_0-\bar\rho\|_{L^2},\bar\rho,u,\|\nabla K\|_{L^{p_0}})$
such that, whenever $A\ge A_0$, problem~\eqref{eq:adv-agg-frac-diff} has a global  solution $\rho^A\in C_b([0,\infty),L^2)\cap C^\infty((0,\infty)\times\mathbb{T}^d)$.
\end{theorem}
\noindent The rough idea of the proof can be described as follows. Recall that aggregation-diffusion equations are generally characterised by two competing forces: the tendency to concentrate due to aggregation versus the tendency to uniformly distribute the initial mass over the spatial domain thanks to diffusion. As long as diffusion dominates, the solution cannot concentrate too much and thus will not blow up. In the delicate case of small diffusion (when the $H^\frac{\gamma}{2}$ norm is not large enough compared to $L^2$) the $\gamma$-RE flow -- if sufficiently strong -- takes care of the low frequencies by quickly stirring the density\footnote{Strictly speaking, this mechanism of stirring only fully applies if $\rho^A(t)$ lies in the continuous spectral subspace corresponding to $U^1$. In the case of a non-trivial component in the $L^2$-closure of the subspace spanned by all (rough) eigenfunctions the mechanism by which gradients are increased is somewhat more technical. The interested reader is referred to~\cite[Lemma~3.3]{constantin_diffusion_2008}.}. This increases spatial gradients, thus enhancing dissipation, and eventually prevents blowup.
\begin{proof}[Proof of Theorem~\ref{thm:suppression} (for $\gamma\le2$ and $2+a-\frac{d}{2}\ge1$)] 
Without loss of generality we may assume that $\rho_0$ is not constant, i.e. $\rho_0\not\equiv\bar\rho$ and $\rho\in C^\infty$ (cf.\,page~\pageref{page:smoothing} (LWP and Smoothing)). By Theorem~\ref{thm:criterion}, it suffices to prove global control in $L^2(\mathbb{T}^d)$. For this purpose we first introduce the following parameters: 
\begin{itemize}
 \item Denote $B:=\|\rho_0-\bar\rho\|_{L^2}>0$.
 \item Let $p_0\in\left(1,\frac{d}{a+1}\right)$, $c_1\ge2$, $b_4$  and $C_\clubsuit(\|\nabla K\|_{L^{p_0}})$ be the constants introduced in the proof of Lemma~\ref{local-control}. We recall that these quantities only depend on $\gamma, a$ and $d$.
 Furthermore denote by $\tau_0=\tau_0(B,\bar\rho,\|\nabla K\|_{L^{p_0}})$ the (possibly small) positive time span~\eqref{eq:timespan} in Lemma~\ref{local-control}. 
\item Define now $\tau_1=\min\left\{\frac{1}{16}\left\{4C_\clubsuit(\|\nabla K\|_{L^{p_0}})
  \left((2B)^{c_1-2}+\bar\rho^{\frac{2}{1-b_4}}\right)\right\}^{-1},\tau_0\right\}$. 
 \item Let $A_0=A_0(\tau_1)$ be such that for any $A\ge A_0$ and any mean-zero $\mu_0\in L^2(\mathbb{T}^d)$ 
 the solution $\tilde\mu^A$ to equation~\eqref{eq:linadvdiff} with initial value $\tilde\mu^A(0)=\mu_0$ saturates the bound $$\|\tilde\mu^A(\tau_1)\|_{L^2}\le \frac{1}{8}\|\mu_0\|_{L^2}.$$  The existence of such an $A_0$ is guaranteed by Theorem~\ref{thm:relaxation}. Obviously, $A_0$ can be chosen to be non-increasing on $\mathbb{R}^+$ and it will necessarily become unbounded near $\tau=0.$ 
\end{itemize}
Now define $t_0=\inf\{t>0:\|\rho^A(t)-\bar\rho\|_{L^2}\ge B\}$. If $t_0=\infty$, there is nothing to prove. We therefore assume $t_0<\infty$ so that by continuity $\|\rho^A(t_0)-\bar\rho\|_{L^2}=B$.  Since $\nabla\cdot(Au)=0$ the statement of Lemma~\ref{local-control} applies to $\rho=\rho^A$, and recalling $\tau_1\le\tau_0$,
we deduce the bound 
\begin{align}\label{eq:2B}
\|\rho^A(t_0+\tau)-\bar\rho\|_{L^2}\le 2B\hspace{.8cm}\text{for all }\tau\in[0,\tau_1].
\end{align}

  In the following we will show that the above choice of $A_0$ implies the bound $\|\rho^A(t_0+\tau_1)-\bar\rho\|_{L^2}\le B$. The claim then follows by iterating the argument: define $t_1=\inf\{t>t_0+\tau_1:\|\rho^A(t)-\bar\rho\|_{L^2}\ge B\}$ and proceed as before with $t_0$ replaced by $t_1$ etc. This then results in the global bound $\|\rho^A(t)-\bar\rho\|_{L^2}\le 2B$ for all $t>0$.
  
Denote $R(\tau)=\int_{t_0}^{t_0+\tau}\|\rho^A\|^2_{\dot H^\frac{\gamma}{2}}$. 
 We distinguish the following cases, which reflect the idea described above.
 
\noindent \textbf{Case I:} $R(\tau_1)>B^2$.
\\
Here we apply estimate~\eqref{eq:usedLater} (with $\rho$ replaced by $\rho^A$), which is possible since $Au$ is divergence-free. 
 Hence on the time interval $[t_0,t_0+\tau_1]$, we have 
\begin{align*}
  \frac{d}{dt}\|\rho^A-\bar\rho\|_{L^2}^2&\le-\|\rho^A\|_{\dot H^\frac{\gamma}{2}}^2+C_\clubsuit(\|\nabla K\|_{L^{p_0}})
  \left(\|\rho^A-\bar\rho\|_{L^2}^{c_1}+\bar\rho^{\frac{2}{1-b_4}}\|\rho^A-\bar\rho\|_{L^2}^{2}\right)
  \\&\le -\|\rho^A\|_{\dot H^\frac{\gamma}{2}}^2+4C_\clubsuit(\|\nabla K\|_{L^{p_0}})
  \left((2B)^{c_1-2}+\bar\rho^{\frac{2}{1-b_4}}\right)B^2,
\end{align*}
where we used~\eqref{eq:2B} in the second step.
We now integrate in time from $t_0$ to $t_0+\tau_1$ to obtain  
\begin{align*}
  \|\rho^A-\bar\rho\|_{L^2}^2(t_0+\tau_1)&\le B^2-B^2+\tau_1\cdot4C_\clubsuit(\|\nabla K\|_{L^{p_0}})\left((2B)^{c_1-2}+\bar\rho^{\frac{2}{1-b_4}}\right)B^2.
  \\&\le \frac{1}{16}B^2
\end{align*}
Here we used the hypothesis (of Case I) and, in the second step, the choice of $\tau_1$.
 \\
 \textbf{Case II:} $R(\tau_1)\le B^2$.
 \\ 
In this case we need to approximate $\rho^A(t_0+t)$ by the solution $\mu^A(t_0+t)$ to equation~\eqref{eq:linadvdiff} with datum $\mu^A(t_0)=\rho^A(t_0)$. We estimate
 \begin{align*}
  \frac{1}{2}\frac{d}{dt}\|\rho^A-\mu^A\|_{L^2}^2&+\|\rho^A-\mu^A\|_{\dot H^\frac{\gamma}{2}}^2 =-\int\rho^A \nabla K\ast\rho^A\cdot\nabla(\rho^A-\mu^A)\\
  &\le \frac{1}{2}\|\rho^A\nabla K\ast\rho^A\|^2_{\dot H^{1-\frac{\gamma}{2}}}+\frac{1}{2}\|\rho^A-\mu^A\|^2_{\dot H^\frac{\gamma}{2}}.
  \end{align*}
Absorption yields
 \begin{align}\label{eq:L2estA}
  \frac{1}{2}\frac{d}{dt}\|\rho^A-\mu^A\|_{L^2}^2&+\frac{1}{2}\|\rho^A-\mu^A\|_{\dot H^\frac{\gamma}{2}}^2\le \frac{1}{2}\|\rho^A\nabla K\ast\rho^A\|^2_{\dot H^{1-\frac{\gamma}{2}}}.
  \end{align}
Thanks to estimate~\eqref{eq:forCaseII}, the RHS of~\eqref{eq:L2estA} is bounded from above by
\begin{align*}
 \frac{1}{2}\left\{\frac{1}{2}\|\rho\|_{\dot H^\frac{\gamma}{2}}^2+C_\clubsuit(\|\nabla K\|_{L^{p_0}})\left(\|\rho-\bar\rho\|_{L^2}^{c_1}+\bar\rho^{\frac{2}{1-b_4}}\|\rho-\bar\rho\|_{L^2}^{2}\right)\right\}.
\end{align*}
Combination with~\eqref{eq:2B} implies on the time interval $[t_0,t_0+\tau_1]$
\begin{align*}
 \frac{d}{dt}\|\rho^A-\mu^A\|_{L^2}^2&+\|\rho^A-\mu^A\|_{\dot H^\frac{\gamma}{2}}^2\le \frac{1}{2}\|\rho\|_{\dot H^\frac{\gamma}{2}}^2+4C_\clubsuit(\|\nabla K\|_{L^{p_0}})\left((2B)^{c_1-2}+\bar\rho^{\frac{2}{1-b_4}}\right)B^2.
  \end{align*}
We now integrate from $t_0$ to $t_0+\tau_1$ to conclude using also the hypothesis (of Case II)
\begin{align*}
 \|\rho^A-\mu^A\|_{L^2}^2(t_0+\tau_1)&\le \frac{1}{2}B^2+\tau_1 \cdot4C_\clubsuit(\|\nabla K\|_{L^{p_0}})\left((2B)^{c_1-2}+\bar\rho^{\frac{2}{1-b_4}}\right)B^2
   \\&\le  \frac{1}{2}B^2+ \frac{1}{16}B^2
   \\&= \frac{9}{16}B^2.
  \end{align*}
In the second step of the last estimate, we used the choice of $\tau_1$.

 Note that as $\mu^A(t_0)-\bar\rho=\rho^A(t_0)-\bar\rho$ (whose $L^2$-norm equals $B$), by choice of $A_0$ and since $A\ge A_0$, the bound 
 $$\|\mu^A-\bar\rho\|_{L^2}(t_0+\tau_1)\le \frac{1}{8}B$$ holds true.
 We therefore obtain
\begin{align*}
 \|\rho^A-\bar\rho\|_{L^2}(t_0+\tau_1)&\le\|\rho^A-\mu^A\|_{L^2}(t_0+\tau_1)+\|\mu^A-\bar\rho\|_{L^2}(t_0+\tau_1)
 \\&\le\frac{7}{8}B
\end{align*}

In any case we obtain
\begin{align*}
  \|\rho^A-\bar\rho\|^2_{L^2}(t_0+\tau_1)&\le B^2,
\end{align*}
which completes the proof.
\end{proof}

\begin{remarks*}$ $\vspace{-.1cm}
\begin{enumerate}[(i)]
 \item 
In the proof of Theorem~\ref{thm:suppression}, we actually obtained a strictly smaller $L^2$-distance from the mean after one iteration step.
By slightly adapting the above proof, one may verify that enhanced relaxation (in the sense of Theorem~\ref{thm:relaxation}) can be obtained also for the non-linear problem~\eqref{eq:adv-agg-frac-diff} if one allows the amplitude $A_0$ to depend on the initial datum~$\rho_0$ (as well as on the parameters of the problem).
 \item Note that for $d=2$ and $a=0$ the kernel $\nabla K$ has the same singularity at the origin as $\nabla N$, where $N$ denotes the two-dimensional Newton kernel. Although on the torus $N$ is not a proper convolution kernel, an analysis almost completely analogous to the one established here, shows that the statement of Theorem~\ref{thm:suppression} also applies to the two-dimensional parabolic-elliptic Keller-Segel model with fractional diffusion $-\Lambda^\gamma$ whenever $\gamma>1$.
 Similarly,  for the three-dimensional parabolic-elliptic Keller-Segel model with fractional diffusion, we have blowup prevention whenever $\gamma>\frac{3}{2}$.
\end{enumerate}
\end{remarks*}

\noindent
Note that for dimension $d\ge4$ Theorem~\ref{thm:suppression} no longer
includes the Keller-Segel case as the lower bound $\gamma_0=d/2$ would enforce diffusion to be stronger than classical (more concretely, it is the fact that the assumption $\frac{d}{2(d-2)}>1$ (cf.~\eqref{eq:L2cond}) is violated what makes our arguments break down). 
As alluded to in the introduction, the reason for this failure is the fact that the $L^2$-norm is no longer subcritical for~Keller-Segel in $d\ge4$.

Scaling suggests that by working in $L^p$ spaces of higher integrability ($p>2$) smaller lower bounds on $\gamma$ may be achieved, namely    
\begin{align}\label{eq:gamma}
  \gamma>2+a-d\left(1-\frac{1}{p}\right)
\end{align}
(as long as $\gamma$ is large enough so that the nonlinear equation is locally well-posed in a suitable Lebesgue (or Sobolev) space and for data of the corresponding regularity Theorem~\ref{thm:relaxation} is valid  for this $\gamma$. The additional condition $\gamma>1$, for instance, would ensure these last two properties).
For the Keller-Segel type (Newton kernel) singularity inequality~\eqref{eq:gamma} becomes $\gamma>\frac{d}{p}$.
This may lead to the expectation that also in the higher-dimensional Keller-Segel model  the mixing mechanism is able to prevent blowup for any $\gamma>1$ when confining to e.g.\;$L^\infty(\mathbb{T}^d)$ initial data. However, when trying to prove suppression using $L^p$- instead of $L^2$-estimates the following issue arises: 
following the notation in the proof of Theorem~\ref{thm:suppression}, it appears that in $L^p$, $p>2$, the approximation of $\rho^A$ by $\mu^A$ requires an estimate of the form
\begin{align}\label{eq:necessaryEst}
  \|\Lambda^{1-\frac{\gamma}{2}}f\|_{L^{p_1}}\lesssim \|\Lambda^\frac{\gamma}{2}(|f|^{\frac{p}{2}})\|_{L^2}^{2/p}
\end{align}
for some $p_1>2$. Certainly such an estimate cannot hold unless  $\gamma>\left(\frac{1}{p}+\frac{1}{2}\right)^{-1}$,  a lower bound which is strictly larger than $1$ if $p>2$. 
Despite the above scaling heuristics and although the weakly mixing/relaxation enhancing condition on the flow appears to be a very strong hypothesis, it is not obvious to the authors how to extend the approach in such a way that it includes the Keller-Segel model with fractional diffusion of any strength {$\gamma>1$} and in any dimension $d\ge2$. 

In the case $\gamma=2$, however, estimate~\eqref{eq:necessaryEst} becomes trivial, and indeed, in this case by working in $L^p$ instead of $L^2$ the suppression mechanism can be extended as to include in particular the classical Keller-Segel model ($\gamma=2$) in any dimension $d\ge2$, which we would like to illustrate in the following.
 Let us consider the Keller-Segel model -- in its precise form for clarity's sake -- 
under the influence of a strong incompressible flow
\begin{equation}\label{eq:classicalKS}
  \partial_t\rho^A+Au\cdot\nabla\rho^A=\triangle\rho^A+\nabla\cdot(\rho^A\nabla \triangle^{-1}(\rho^A-\bar\rho)) \;\text{ in }\;(0,\infty)\times\mathbb{T}^d
\end{equation}
with $d\ge4$.
The higher-dimensional Keller-Segel model with standard diffusion (i.e.\,\eqref{eq:classicalKS} with $A=0$) is $L^\frac{d}{2}$-critical and $L^1$-supercritical (choose $\gamma=2$, $a=d-2$ in~\eqref{eq:scale}). For $p>\frac{d}{2}$ local well-posedness in $L^p$  and regularity for positive times are well-established in the community (see e.g.~\cite{biler_existence_1994} for results on bounded domains and~\cite{calvez_blow-up_2012} for results on the whole space assuming sufficient decay at infinity), and at any (positive) level of mass (= $L^1$-norm for non-negative solutions) there exist smooth solutions  which blow up in finite time~\cite{biler_existence_1994, biler_blowup_2010, calvez_blow-up_2012}.
Moreover, for global regularity it suffices to globally control the $L^p$-norm of the solution, 
and statements analogous to those established in Section~\ref{sec:apriori} hold true whenever $p>\frac{d}{2}$. 
We will therefore directly proceed to the proof of global regularity for~\eqref{eq:classicalKS} whenever $A$ is sufficiently large.

\begin{theorem}[Prevention of blowup for Keller-Segel model in higher dimensions]\label{thm:suppKS}
 Assume $d\ge6$ and let $p>\frac{d}{2}$.  Suppose that the divergence-free smooth vector field $u(x)$ is $2$-relaxation enhancing. Then for any initial datum $\rho_0\in L^p(\mathbb{T}^d)$ there exists $A_0(\|\rho_0-\bar\rho\|_{L^p},\bar\rho,u,p)$ 
such that, whenever $A\ge A_0$, equation~\eqref{eq:classicalKS} has a global solution $\rho^A\in C_b([0,\infty),L^p)\cap C^\infty((0,\infty)\times\mathbb{T}^d)$ with initial value $\rho^A(0)=\rho_0$.\\
For $d=4,5$ the statement holds true under the stronger condition $p>\frac{4d}{d+2}$.
\end{theorem}
\begin{rem}
For $d\ge6$ Theorem~\ref{thm:suppKS} is optimal in terms of the regularity required for the initial data in the sense that equation~\eqref{eq:classicalKS} with $A=0$ is $L^\frac{d}{2}$-critical.
\end{rem}

\begin{proof}[Proof of Theorem~\ref{thm:suppKS}]The result follows from arguments similar to Theorem~\ref{thm:suppression} with $L^2$ replaced by $L^p$. In contrast to the proof of Theorem~\ref{thm:suppression}, here we do not (need to) distinguish the cases of small and large diffusion: 
 for any time $t_0\ge0$ -- even if diffusion is large -- the local solution $\rho^A(t_0+\tau)$ to 
  \eqref{eq:classicalKS} can be approximated sufficiently well by 
 the solution $\mu^A(t_0+\tau)$ to equation~\eqref{eq:linadvdiff} with datum $\mu^A(t_0)=\rho^A(t_0)$ for small enough times $\tau>0$, as will be shown in the following. 
 
 We first prove the case $d\ge6$. Without loss of generality we may assume $p<d$. 
    Note that since $p<d$ we can define $q\in(p,\infty)$ via
  \begin{align}\label{eq:def-q}
    \left(\frac{1}{p}-\frac{1}{q}\right)d=1.
  \end{align}
  Since $d\ge6$ and $p>\frac{d}{2}$, we have
  \begin{align*}
    \frac{1}{p}+\frac{1}{q}=\frac{2}{p}-\frac{1}{d}<\frac{3}{d}\le\frac{1}{2}
  \end{align*}
so that there exists  $r\in(2,\infty)$ 
  satisfying
  \begin{align*}
   \frac{1}{p}+\frac{1}{q}+\frac{1}{r}=\frac{1}{2}.
  \end{align*}
 We now let $h=|\rho^A-\mu^A|^{p/2}$ and estimate using equation~\eqref{eq:classicalKS} and $\nabla\cdot u=0$
 \begin{align*}
  \frac{1}{p}\frac{d}{dt}\|\rho^A-\mu^A\|_{L^p}^p&+\frac{4(p-1)}{p^2}\|\nabla h\|_{L^2}^2 
  \\&\le-\int\rho^A\nabla \triangle^{-1}(\rho^A-\bar\rho)\cdot\nabla\left((\rho^A-\mu^A)|\rho^A-\mu^A|^{p-2}\right)
  \\&\le C\|\rho^A\|_{L^p} \|\nabla \triangle^{-1}(\rho^A-\bar\rho)\|_{L^{q}} \|\rho^A-\mu^A\|_{L^{r(p/2-1)}}^{p/2-1}\|\nabla h\|_{L^2}
  \\&\le C\|\rho^A\|_{L^p} \|\rho^A-\bar\rho\|_{L^p} \|h\|_{L^{r_1}}^{(p-2)/p}\|\nabla h\|_{L^2},
  \end{align*}
  where $r_1$ is defined via
\begin{align*}
 r_1\cdot p/2&=r(p/2-1).
\end{align*}
  In the last estimate we used Lemma~\ref{l:Sobolev} (exploiting our choice of $q$) and the boundedness of the Riesz transform on $L^p,\, p\in(1,\infty)$.
For $p\in(\frac{d}{2},d)$ and $d\ge6$ an elementary check yields $r_1>2.$ (Of course, $r_1\in[1,2]$ would be even easier.)
Now note that by Lemmata~\ref{l:Sobolev} and~\ref{l:interpol} 
for $\sigma=\left(\frac{1}{2}-\frac{1}{r_1}\right)d$ we have
\begin{align*}
 \|h\|_{L^{r_1}}\lesssim \|\Lambda^\sigma h\|_{L^2}\lesssim \|\nabla h\|_{L^2}^\sigma\|h\|_{L^2}^{1-\sigma}.
\end{align*}
Hence we obain
 \begin{align}\label{eq:LpEstimate}
 \frac{d}{dt}\|\rho^A-\mu^A\|_{L^p}^p&+\|\nabla h\|_{L^2}^2 \nonumber
  \\&\le C(\|\rho^A-\bar\rho\|_{L^p}+\bar\rho) \|\rho^A-\bar\rho\|_{L^p} \|h\|_{L^2}^{(1-\sigma)(p-2)/p}\|\nabla h\|_{L^2}^{1+\sigma(p-2)/p}.
  \end{align}
  It is elementary to verify that $p>\frac{d}{2}$ guarantees
 $\sigma(p-2)/p<1.$
  Thus an absorption argument yields
  \begin{align}\label{eq:LpEst2}
  \frac{1}{p}\frac{d}{dt}\|\rho^A-\mu^A\|_{L^p}^p
  &\le C(\|\rho^A-\bar\rho\|_{L^p}+\bar\rho)^{c_3} \|\rho^A-\bar\rho\|_{L^p}^{c_3} \|h\|_{L^2}^{c_4}
  \end{align}
  with $c_i=c_i(\sigma,p), i=3,4$, suitable positive exponents.
  Similarly to Lemma~\ref{local-control}, for $B:=\max\{\|\rho^A(t_0)-\bar\rho\|_{L^p},1\}$ 
  one can show\footnote{Since for the Keller-Segel model this is a well-known result, its proof is omitted here. Of course, the condition $p>\frac{d}{2}$ is crucial for its validity.} that $\|\rho^A-\bar\rho\|_{L^p}\le 2B$ on some small time interval $[t_0,t_0+\tau_0]$ where $\tau_0>0$ only depends on $B,\bar\rho$ and fixed parameters. 
  Also notice that on $[t_0,t_0+\tau_0]$ we then have  $\|h\|_{L^2}=\|\rho^A-\mu^A\|_{L^p}^{p/2}$ and $\|\rho^A-\mu^A\|_{L^p}\le \|\rho^A-\bar\rho\|_{L^p}+ \|\mu^A-\bar\rho\|_{L^p}\le 3B$, where in the last bound we used the fact that $\|\mu^A-\bar\rho\|_{L^p}$ is non-increasing on $[t_0,\infty)$.
The rest of the argument is similar to the reasoning in Case~II of the proof of Theorem~\ref{thm:suppression} except that here we need to use Remark~\ref{rem:gamma<1}\,\eqref{rem:Lp-relax} instead of Theorem~\ref{thm:relaxation}.

If $d=4,5$ we assume again without loss of generality $p<d$ and define $q$ via~\eqref{eq:def-q}. The condition $p>\frac{4d}{d+2}$ ensures that $\frac{1}{p}+\frac{1}{q}<\frac{1}{2}$. The rest of the proof then follows as before.
\end{proof}

\begin{appendix}
\section{Blowup in the absence of advection}\label{app:blowup}
In this section we aim to show that in the case $a=0$ and in the absence of strong advection there exist smooth initial data which lead to blowup in finite time.
We stress that blowup can also be produced in the presence of the advective term if one \textit{first} fixes the flow $Au$ (including its amplitude) and choses appropriate data \textit{afterwards}.

We consider the equation 
\begin{align}\label{eq:agg-diff}
 \partial_t\rho=-\Lambda^\gamma\rho+\nabla\cdot(\rho\nabla K\ast\rho)\;\;\text{ in }\;(0,\infty)\times\mathbb{T}^d,
\end{align}
where $\nabla K(x)\sim\frac{x}{|x|^2}$ near $x=0$, $d\ge2$ and $\gamma\in(1,2]$.
In this case, blowup can be produced by a construction very similar to the one in \cite{kiselev_suppression_2016}. We therefore confine ourselves to sketching the main argument and indicating the steps which deviate from~\cite{kiselev_suppression_2016}.
Let us introduce the following parameters and auxiliary functions:
\begin{itemize}
\item $0<2a<b<\frac{1}{8}$ (sufficiently small).
 \item $\rho_0\in C^\infty(\mathbb{T}^d)$ nonnegative with $\supp\rho_0\subset B_a(0)$ and mass $M\ge1$ (sufficiently large).
 \item $\phi$ a smooth cut-off at scale $b$: Fix $\phi_0\in C^\infty(\mathbb{R}^d)$ with $\supp\phi_0\subset B_{1}$,
 $\phi_0\equiv1$ on $B_{\frac{1}{2}}$, $0\le\phi_0\le1$. Then $\phi(x):=\phi_0(\frac{x}{b})$ can be considered as a function on the periodic box $\mathbb{T}^d$.
\end{itemize}
For simplicity we assume equality $\nabla K(x)=\frac{x}{|x|^2}$ on $B_{\frac{1}{4}}$.
The parameters $a,b,M$ will be fixed later. As long as the solution $\rho$ stays regular, it preserves positivity and mass. 

The main ingredient in the blowup proof is a virial argument, which can be exploited when considering the evolution of the second moment.
This is a standard technique for proving blowup of the two- and higher-dimensional Keller-Segel model in bounded domains and the whole space.
\begin{lemma}[Decrease of $2^{\text{nd}}$ moment]\label{l:2ndmom} Let $T>0$ and assume that problem~\eqref{eq:agg-diff} subject to initial condition  $\rho(0)=\rho_0$ has a regular solution $\rho$ 
 on $[0,T]$. Then for all $t\in[0,T]$ 
\begin{align*}
  \frac{d}{dt}\int_{\mathbb{T}^d}|x|^2\rho(t,x)\phi(x)\,dx&\le - \left(\int\rho(t,x)\phi(x)\,dx\right)^2+C_2M\|\rho(t,\cdot)\|_{L^1(\mathbb{T}^d\setminus B_b)}
  \\&\phantom{\le}\;+C_3bM^2+C_4M.
\end{align*}
\end{lemma}
\begin{remark*}
 Note that since $\supp \phi\subset(-\frac{1}{2},\frac{1}{2})$ the integrand on the LHS is well-defined on the periodic box~$\mathbb{T}^d$.
\end{remark*}

\begin{proof}
We compute 
\begin{align*}
  \frac{d}{dt}\int_{\mathbb{T}^d}|x|^2\rho(t,x)\phi(x)\,dx&=-\int_{\mathbb{T}^d}\rho(t,x)\Lambda^\gamma(|x|^2\phi(x))\,dx
  \\&\phantom{=}-\int_{\mathbb{T}^d}\int_{\mathbb{T}^d}\nabla(|x|^2\phi(x))\cdot\nabla K(x-y)\rho(t,y)\rho(t,x)\,dydx
  \\&=: (i)+(ii).
 \end{align*}
In order to estimate the first term on the RHS, let us recall that for $\gamma\in(0,2)$ the fractional Laplacian has the following representation (see e.g.\,\cite{cordoba_maximum_2004} or~\cite{roncal_2016}): 
\begin{align*}
 \Lambda^\gamma f(x)=\text{p.v.}\int_{\mathbb{T}^d}(f(x)-f(y))G_{\gamma,d}(x-y)dy,
\end{align*}
where
\begin{align*}
 G_{\gamma,d}(z)=c_{\gamma,d}\sum_{\alpha\in\mathbb{Z}^d}\frac{1}{|z-\alpha|^{d+\gamma}}, \;\; z\neq0,
\end{align*}
and $c_{\gamma,d}$ is a normalisation constant. 
Using the above formula and the smoothness of $\phi_0$, it is easy to see that there exists a positive constant $C_{\phi_0}<\infty$ such that for all $b\in(0,1]$ 
  \begin{align*}
  \left\|\Lambda^{\gamma}\left(|x|^2\phi_0\left(\frac{x}{b}\right)\right)\right\|_{L^\infty(\mathbb{T}^d)}&\le C_{\phi_0}b^{2-\gamma}.
 \end{align*}
 Recalling $\phi(x)=\phi_0(\frac{x}{b})$, we conclude $(i)\le CMb^{2-\gamma}$.\\
 To estimate the second term, we introduce the splitting
 \begin{align*}
  (ii)&=-2\int_{\mathbb{T}^d}\int_{\mathbb{T}^d}\phi(x)x\cdot\nabla K(x-y)\rho(t,y)\rho(t,x)\,dydx
  \\&\phantom{=}-\int_{\mathbb{T}^d}\int_{\mathbb{T}^d}|x|^2\nabla\phi(x)\cdot\nabla K(x-y)\rho(t,y)\rho(t,x)\,dydx
  \\&=-2\int_{\mathbb{T}^d}\int_{\mathbb{T}^d}\phi(x)x\cdot\nabla K(x-y)\rho(t,y)\rho(t,x)\phi(y)\,dydx
  \\&\phantom{=}-2\int_{\mathbb{T}^d}\int_{\mathbb{T}^d}\phi(x)x\cdot\nabla K(x-y)\rho(t,y)\rho(t,x)(1-\phi(y))\,dydx
  \\&\phantom{=}-\int_{\mathbb{T}^d}\int_{\mathbb{T}^d}|x|^2\nabla\phi(x)\cdot\nabla K(x-y)\rho(t,y)\rho(t,x)\,dydx
  \\&=:(iii)+(iv)+(v).
 \end{align*}
 On $\{x-y:x,y\in\supp\phi\}$ we have $\nabla K(z)=\frac{z}{|z|^2}.$ Thus, upon symmetrisation,
\begin{align*}
 (iii)&=-\int_{\mathbb{T}^d}\int_{\mathbb{T}^d}\frac{|x|^2-2x\cdot y+|y|^2}{|x-y|^2}\phi(x)\rho(t,y)\rho(t,x)\phi(y)\,dydx
 \\&=-\left(\int_{\mathbb{T}^d}\rho(t,x)\phi(x)\,dx\right)^2.
\end{align*}
Next, we note 
\begin{align*}
 (iv)&=-2\int_{\mathbb{T}^d}\int_{\mathbb{T}^d}\phi(x)x\cdot\nabla K(x-y)\rho(t,y)\rho(t,x)(1-\phi(y))\,dydx
 \\&=-2\int_{\mathbb{T}^d}\int_{\mathbb{T}^d}\phi(x)x\cdot\frac{x-y}{|x-y|^2}\,\chi_{B_\frac{1}{4}}(x-y)\rho(t,y)\rho(t,x)(1-\phi(y))\,dydx
 \\&\phantom{=}+2\int_{\mathbb{T}^d}\int_{\mathbb{T}^d}\phi(x)x\cdot\nabla K(x-y)\chi_{\mathbb{T}^d\setminus B_\frac{1}{4}}(x-y)\rho(t,y)\rho(t,x)(1-\phi(y))\,dydx
 \\&=-\int_{\mathbb{T}^d}\int_{\mathbb{T}^d}\left[\phi(x)(1-\phi(y))x-\phi(y)(1-\phi(x))y\right]\cdot\frac{x-y}{|x-y|^2}\,\chi_{B_\frac{1}{4}}(x-y)\rho(t,y)\rho(t,x)\,dydx
 \\&\phantom{=}+\int_{\mathbb{T}^d}\int_{\mathbb{T}^d}\phi(x)x\cdot\nabla K(x-y)\chi_{\mathbb{T}^d\setminus B_\frac{1}{4}}(x-y)\rho(t,y)\rho(t,x)(1-\phi(y))\,dydx
  \\&\le CM\|\rho(t)\|_{L^1(\mathbb{T}^d\setminus B_\frac{b}{2})}+CbM^2.
\end{align*}
In the last step we used $$|\left[\phi(x)(1-\phi(y))x-\phi(y)(1-\phi(x))y\right]|\le C\chi_{\mathbb{T}^d\times\mathbb{T}^d\setminus B_\frac{b}{2}\times B_\frac{b}{2}}(x,y).$$
Similar arguments yield
\begin{align*}
(v)&=-\int_{\mathbb{T}^d}\int_{\mathbb{T}^d}|x|^2\nabla\phi(x)\cdot\nabla K(x-y)\rho(t,y)\rho(t,x)\,dydx
\\&\le CM\|\rho(t)\|_{L^1(\mathbb{T}^d\setminus B_\frac{b}{2})}+CbM^2. 
\end{align*}
(In both estimates (and thus also in claimed estimate) the term $CbM^2$ can actually be dropped.)
Using all these estimates, we conclude
\begin{align*}
  \frac{d}{dt}\int_{\mathbb{T}^d}|x|^2\rho(t,x)\phi(x)\,dx&\le-\left(\int_{\mathbb{T}^d}\rho(t,x)\phi(x)\,dx\right)^2+CM\|\rho(t)\|_{L^1(\mathbb{T}^d\setminus B_\frac{b}{2})}\\&\phantom{\le}\;+CM^2b+CMb^{2-\gamma}. 
 \end{align*}
Since $\gamma\le2$, the claimed bound follows.
\end{proof}

Next, we need to ensure that the mass -- initially localised near the origin -- cannot escape too fast. The statement and proof are analogous to~\cite[Lemma~8.3]{kiselev_suppression_2016},
where the extension to $\gamma\in(1,2]$ follows as in the previous lemma. 

The existence of exploding solutions is shown completely analogously to~\cite[Proof of Theorem~8.1]{kiselev_suppression_2016}.

\section{Transport-diffusion equation}\label{app:tranport-diffusion}
In this section we will prove Theorem~\ref{thm:relaxation} in the remaining case $\gamma\in[1,2)$.
The proof of this theorem follows along the lines of the proof of \cite[Theorem 1.4]{constantin_diffusion_2008}, and we therefore only point out the differences.
First of all, if $\gamma<2$, condition~(2.1) in \cite{constantin_diffusion_2008} is no longer satisfied. We have the following replacement for~\cite[Theorem 2.1]{constantin_diffusion_2008}. 
\begin{theorem}[Local well-posedness]\label{thm:lwp} Assume $\gamma\in(1,2)$ and let $v=v(x)$ be a smooth divergence-free vector field. For any $T>0$ and $\mu_0\in H^\frac{\gamma}{2}(\mathbb{T}^d)$ there exists a unique solution $$\mu\in L^2(0,T;H^\gamma)\cap C([0,T];H^\frac{\gamma}{2})\text{ with }\partial_t\mu\in L^2(0,T;L^2)$$ of the Cauchy problem
  \begin{align}\label{eq:lwp}
   \partial_t\mu+v\cdot\nabla \mu&=-\Lambda^\gamma\mu\;\;\text{ in }\;(0,T)\times\mathbb{T}^d,\\
   \mu(0)&=\mu_0.\nonumber
  \end{align}
\end{theorem}
\begin{proof}The existence of weak solutions
\begin{equation}\label{eq:uniq_reg}
  \mu\in L^2(0,T; H^\frac{\gamma}{2})\cap C([0,T];L^2)\text{ with }\partial_t\mu\in L^2(0,T; H^{-(1-\frac{\gamma}{2})})
  \end{equation}
  to initial datum $\mu_0\in L^2(\mathbb{T}^d)$ can be shown via a simple Galerkin scheme. 
Since $\gamma>1$, regularity and uniqueness are straightforward as well.
\end{proof}
\begin{rem}
  If $\gamma\in(0,1]$ local existence and uniqueness of weak solutions $\mu\in C([0,T];H^\frac{1}{2})$ with $\partial_t\mu\in C([0,T];H^{-\frac{1}{2}})$ to the Cauchy problem~\eqref{eq:lwp} with initial datum in $H^\frac{1}{2}$ can still be established: the existence of rough solutions is again obtained via a Galerkin method. To prove the claimed regularity and uniqueness one first notes that the constructed weak solution $\mu$ satisfies the pointwise equality
  \begin{align*}
   \partial_tS_k\mu+\nabla\cdot S_k(v\mu)=-\Lambda^\gamma S_k\mu,
  \end{align*}
where $S_k$ are the LP-projections introduced in~Appendix~\ref{app:transport}, and then proceeds as in the proof of
Proposition~\ref{prop:bounded-evol}.
\end{rem}

Owing to the worse regularity, more care has to be taken when approximating 
the advection-diffusion equation by the pure transport equation.
Our replacement for~\cite[Lemma 2.4]{constantin_diffusion_2008} is the following 
\begin{lemma}[Approximation by pure transport]\label{l:approx-by-transport}Let $v=v(x)$ be a smooth divergence-free vector field. Assume $\gamma\in[1,2)$ and let $\eta_0\in H^\frac{\gamma}{2}(\mathbb{T}^d)$. 
Let $\eta^0\in C([0,\infty);H^\frac{\gamma}{2})$ be a weak solution of the transport problem~\eqref{eq:lin-tp} and let $\eta^\varepsilon=\mu$ solve~\eqref{eq:lwp} with $-\Lambda^\gamma$ replaced~\footnote{In order to facilitate the comparison with~\cite{constantin_diffusion_2008}, we adopt the rescaling to \say{small diffusion on long time scales} as introduced in~\cite{constantin_diffusion_2008}.} by $-\varepsilon\Lambda^\gamma$ and initial datum $\eta_0$. Then
\begin{equation}\label{eq:approx}
 \frac{d}{dt}\|\eta^\varepsilon(t)-\eta^0(t)\|^2_{L^2}\le\frac{\varepsilon}{2}\|\eta^0(t)\|_{\dot H^{\gamma/2}}^2\le \frac{\varepsilon}{2}\exp(C(v)t)\|\eta_0\|_{\dot H^{\gamma/2}}^2,
\end{equation}
where $C(v)$ is the constant from Proposition~\ref{prop:bdd-evol-1}.
\end{lemma}
\begin{proof}
 The difference $\eta^\varepsilon-\eta^0$ satisfies
 \begin{equation}\label{eq:dualpair}
  \partial_t(\eta^\varepsilon-\eta^0)+u\cdot\nabla(\eta^\varepsilon-\eta^0)=-\varepsilon\Lambda^\gamma\eta^\varepsilon,
 \end{equation}
 where for fixed time $t$ the equality is to be understood
  in $H^{\frac{\gamma}{2}-1}\subseteq H^{-\frac{\gamma}{2}}$. We can therefore take the dual pairing  $\dot H^{-\frac{\gamma}{2}} \times \dot H^{\frac{\gamma}{2}}$
of the equation with $(\eta^\varepsilon-\eta^0)(t)\in  H^\frac{\gamma}{2}$ 
to obtain after an absorption argument the first inequality in~\eqref{eq:approx}.
 (Here we also used the incompressibility and the smoothness of the flow which guarantee that $B(f,g):=\langle u\cdot\nabla f,g\rangle_{H^{-\frac{1}{2}},H^\frac{1}{2}}$ satisfies $B(f,f)=0$ for all $f\in C^\infty$ and extends uniquely to a bounded bilinear form on $H^\frac{1}{2}\times H^\frac{1}{2}$.)
The second inequality in~\eqref{eq:approx} is just the boundedness of the transport evolution with respect to  $\|\cdot\|_{\dot H^\frac{\gamma}{2}}$ (cf.~\eqref{eq:transport-bound} or Appendix~\ref{app:transport}).
\end{proof}
\begin{rem}
 The statement of Lemma~\ref{l:approx-by-transport} remains true for $\gamma\in(0,1)$ if restricting to initial data in $H^\frac{1}{2}$. Indeed, in this case one only needs to notice that (for fixed time) the equation~\eqref{eq:dualpair} holds in $H^{-\frac{1}{2}}$ and that $(\eta^\varepsilon-\eta^0)(t)\in  H^\frac{1}{2}$.
\end{rem}
The remaining lemmata used in the proof of~\cite[Theorem 1.4]{constantin_diffusion_2008} can either be shown by similar arguments as in Lemma~\ref{l:approx-by-transport} (where for mere $L^2$ data the regularity~\eqref{eq:uniq_reg} has to be used) or require only a formal adaptation (such as replacing the \say{diffusion operator} $-\Gamma$ by $-\Lambda^\gamma$).

     \section{Transport equation in \texorpdfstring{$H^\sigma(\mathbb{T}^d)$}{}    }\label{app:transport}
Here we are concerned with the linear transport equation with a (prescribed) divergence-free smooth velocity field $v=v(x)$:
\begin{align}\label{lin_transport}
 \partial_t\eta+v\cdot\nabla\eta&=0\;\;\text{ in }\;(0,\infty)\times\mathbb{T}^d,\\ 
 \eta(0)&=\eta_0\nonumber.
\end{align} 
Our aim is to prove Proposition~\ref{prop:bdd-evol-1}, i.e.\,the boundedness of the associated evolution in fractional Hilbert spaces $H^\sigma(\mathbb{T}^d), \sigma>0$, where we do not aim for optimal regularity with respect to $v$. In the whole space case fairly general a priori estimates in Besov spaces can be found in~\cite{bahouri_fourier_2011}. As in~\cite{bahouri_fourier_2011} we will make use of a standard tool from harmonic analysis, which we shall introduce in the following.
\subsection{Preliminaries}
We consider a Littlewood-Paley decomposition: let $\phi_0\in C^\infty_c(\mathbb{R}^d)$ be a radial bump function with $\supp\phi_0\subset B_{11/10}(0)$ which is equal to $1$ on $B_{1}(0)$ and satisfies $0\le\phi_0\le1$. Denoting $\phi(\xi):=\phi_0(\xi)-\phi_0(2\xi)$, we then have 
\begin{equation*}
 \phi_0(2\xi)+\sum_{k\ge0}\phi(2^{-k}\xi)=1,\hspace{1cm}\xi\in\mathbb{R}^d.
\end{equation*}
For smooth functions $\eta$ on $\mathbb{T}^d$ we then define the operators 
\begin{equation*}
 S_{-1}\eta(x)=\sum_{\alpha\in\mathbb{Z}^d}\phi_0(2\alpha)\hat\eta(\alpha) \text{e}^{2\pi\text{i}x\cdot\alpha}=\phi_0(0)\hat\eta(0)
\end{equation*} 
and for $k\ge0$
\begin{equation*}
 S_k\eta(x)=\sum_{\alpha\in\mathbb{Z}^d}\phi(2^{-k}\alpha)\hat\eta(\alpha) \text{e}^{2\pi\text{i}x\cdot\alpha}.
\end{equation*}
Note that $S_k$ localises to frequency~$\sim2^k$, i.e. $\supp \widehat{S_k\eta}\subset\{\alpha\in\mathbb{Z}^d:|\alpha|\approx 2^k\}$ and we have equivalence of (semi-) norms
\begin{align}\label{eq:norm-equiv}
  \|\eta\|_{\dot H^\sigma}^2\sim\sum_{k\ge0}2^{2\sigma k}\|S_k\eta\|_{L^2}^2.
\end{align}
We will at times also use the notation $S_{\le N}, S_{M<\dots< N}$ and  $S_{\ge N}$ to denote the sums of operators  corresponding to $\sum_{-1\le k\le N}S_k, \sum_{M< k< N}S_k$ and $\sum_{k\ge N}S_k$. 

\subsection{Boundedness of evolution}
We will now provide a proof of the transport estimate:
\begin{proposition}\label{prop:bounded-evol}Assume $\sigma>0$.
Any sufficiently regular solution $\eta$ of~\eqref{lin_transport} satisfies
 \begin{align}
  \|\eta(t)\|^2_{\dot H^\sigma}\le \exp(C(v)t)\|\eta_0\|^2_{\dot H^\sigma},\hspace{.5cm}t\ge0,
 \end{align}
where the positive constant $C(v)$ saturates the bound $$C(v)\lesssim_{\sigma,d} \|\Lambda^{\sigma+\frac{d}{2}+1}v\|_{L^2}.$$ 
\end{proposition}
The proof exploits the following gain at level $k$ for the commutator involving an LP projection $S_k$ for $k\gg1$.
\begin{lemma}\label{commutator_gain_slow_varying}
For smooth functions $f,g$ on the torus the following commutator estimate holds 
 \begin{equation}
  \|[S_k,g]f\|_{L^2(\mathbb{T}^d)}\le 2^{-k}\|\nabla\phi\|_{L^\infty}\|\hat g(\beta)\beta\|_{l^1_\beta}\|f\|_{L^2(\mathbb{T}^d)}.
 \end{equation} 
\end{lemma}
\begin{proof}[Proof of Lemma~\ref{commutator_gain_slow_varying}]
 We first note 
\begin{equation}
  \|[S_k,g]f\|_{L^2(\mathbb{T}^d)}=\|\widehat{[S_k,g]f}\|_{l^2(\mathbb{Z}^d)}
\end{equation}
and therefore consider 
\begin{align*}
 \widehat{[S_k,g]f}(\alpha)&=\widehat{S_k(gf)}(\alpha)-\widehat g\ast\widehat{S_kf}(\alpha)\\
 &=\sum_{\beta\in\mathbb{Z}^d}[\phi(2^{-k}\alpha)-\phi(2^{-k}(\alpha-\beta))]\hat g(\beta) \hat f(\alpha-\beta)\\
 &=\sum_\beta 2^{-k}\int_0^1\nabla\phi(2^{-k}(\alpha-(1-s)\beta))ds\cdot\beta \hat g(\beta) \hat f(\alpha-\beta).
\end{align*}
Hence 
\begin{align*}
 |\widehat{[S_k,g]f}(\alpha)|&\le2^{-k}\|\nabla\phi\|_{L^\infty(\mathbb{R}^d)}\sum_\beta|\beta \hat g(\beta)|| \hat f(\alpha-\beta)|. 
\end{align*}
Young's convolution inequality then yields the claim
\begin{align*}
 \|\widehat{[S_k,g]f}\|_{l^2(\mathbb{Z}^d)}\le 2^{-k}\|\nabla\phi\|_{L^\infty(\mathbb{R}^d)}\|\beta \hat g(\beta)\|_{l^1_\beta}\|f\|_{L^2},
\end{align*}
where we used $\|\hat f\|_{l^2}=\|f\|_{L^2}$.
\end{proof}
We are now in a position to show the boundedness of the evolution~\eqref{lin_transport} in $\dot H^\sigma(\mathbb{T}^d)$.
\begin{proof}[Sketch proof of Proposition~\ref{prop:bounded-evol}]
Without loss of generality we can assume $\hat\eta(0)=0$. 
In the following we will omit any possible dependence of constants on $\sigma$ and $d$.
Now let $k\ge0$ be a fixed but arbitrary integer. The equation implies
 \begin{equation*}
  \partial_tS_k\eta=-\nabla\cdot S_k(v\eta)
 \end{equation*}
and hence
\begin{equation*}
\frac{1}{2}\frac{d}{dt}\|S_k\eta\|_{L^2(\mathbb{T}^d)}^2=\int -\nabla\cdot S_k(v\eta)\,S_k\eta.
\end{equation*}
Since by incompressibility
\begin{equation*}
 \int \nabla\cdot (v\,S_k\eta)\,S_k\eta=-\frac{1}{2} \int v\cdot\nabla|S_k\eta|^2=0,
\end{equation*}
it follows that 
\begin{align}\label{eq:Sk}
 \frac{1}{2}\frac{d}{dt}\|S_k\eta\|_{L^2(\mathbb{T}^d)}^2&=\int \nabla\cdot [v,S_k]\eta\,S_k\eta\\\nonumber
 &=\int \nabla\cdot \tilde S_k[v,S_k]\eta\,S_k\eta\hspace{2cm}(\tilde S_k\,S_k=S_k)
 \\&\le\|\nabla\cdot \tilde S_k[v,S_k]\eta\|_{L^2}\|S_k\eta\|_{L^2}\nonumber
 \\&\lesssim 2^{k}\|\tilde S_k[v,S_k]\eta\|_{L^2}\|S_k\eta\|_{L^2},\nonumber
\end{align}
where $\tilde S_k$ denotes a suitable Fourier multiplier localising to frequency~$\sim2^k$ whose symbol is equal to $1$ on $\supp \phi(2^{-k}\cdot)$. We now assume $k\gg1$ and split
\begin{align*}
 v=S_{\le k-4}v+S_{> k-4}v
\end{align*}
and consider
\begin{align}\label{eq:split}
 \tilde S_k[v,S_k]\eta=\tilde S_k[S_{\le k-4}v,S_k]\eta+\tilde S_k[S_{> k-4}v,S_k]\eta.
\end{align}
With regard to the regularity of $\eta$, the first term is the delicate one. It can be estimated using Lemma~\ref{commutator_gain_slow_varying}, as we will show now.
Note that there exists a multiplier $S_k'$ localising to frequency~$\sim2^k$ such that
\begin{align*}
  \tilde S_k[S_{\le k-4}v,S_k]\eta=\tilde S_k[S_{\le k-4}v,S_k]S_k'\eta.
\end{align*}
Now Lemma~\ref{commutator_gain_slow_varying} applied to $g=S_{\le k-4}v, f=S_k'\eta$ yields
\begin{align*}
 \|\tilde S_k[S_{\le k-4}v,S_k]S_k'\eta\|_{L^2}&\le\|[S_{\le k-4}v,S_k] S_k'\eta\|_{L^2}
 \\&\le C 2^{-k}\|\widehat{S_{\le k-4}v}(\alpha)\alpha\|_{l^1_\alpha}\|S_k'\eta\|_{L^2}
 \\&\le C 2^{-k}\|\hat{v}(\alpha)\alpha\|_{l^1_\alpha}\|S_k'\eta\|_{L^2},
\end{align*}
where in the last step we used 
\begin{align*}
\|\widehat{S_{\le k-4}v}(\alpha)\alpha\|_{l^1_\alpha}=\sum_\alpha|\sum_{j\le k-4}\phi(2^{-j}\alpha)\hat v(\alpha)\alpha|\le\sum_\alpha|\hat v(\alpha)\alpha|.
\end{align*}
Finally notice that by the equivalence of norms~\eqref{eq:norm-equiv}
\begin{align*}
 \sum_{k\gg1}2^{2k\sigma}2^k\left(2^{-k}\|\hat{v}(\alpha)\alpha\|_{l^1_\alpha}\|S_k'\eta\|_{L^2}\right)\|S_k\eta\|_{L^2}\lesssim \|\hat{v}(\alpha)\alpha\|_{l^1_\alpha}\|\eta(t)\|_{\dot H^\sigma(\mathbb{T}^d)}^2.
\end{align*}

Estimating the second term in~\eqref{eq:split} is straightforward if one is not interested in optimal regularity results for $v$. For a rough estimate, we note that the part of this term which requires the highest regularity of $v$ is 
\begin{align*}  S_k(S_{k-4<\dots< k+4}v\,\eta)
\end{align*} 
as it may involve low frequencies of $\eta$. We first estimate using a Bernstein inequality (see e.g.~\cite[Lemma~2.1]{bahouri_fourier_2011}, which can be proved in a similar way on $\mathbb{T}^d$)
\begin{align*}
 \|S_k(S_{k-4<\dots< k+4}v\,\eta)\|_{L^2}&\lesssim 2^{\frac{kd}{2}}\|S_k(S_{k-4<\dots< k+4}v\, \eta)\|_{L^1}
 \\&\lesssim  2^{\frac{kd}{2}}\|S_{k-4<\dots< k+4}v\|_{L^2}\| \eta\|_{L^2}
\end{align*}
and note that thanks to Cauchy-Schwarz and $\hat\eta(0)=0$ 
\begin{align*}
 \sum_{k\gg1}2^{2k\sigma}2^k&\left(2^{\frac{kd}{2}}\|S_{k-4<\dots< k+4}v\|_{L^2}\| \eta\|_{L^2}\right)\|S_k\eta\|_{L^2}\\&\lesssim  \sum_{k\gg1}\|S_{k-4<\dots< k+4}(\Lambda^{\sigma+\frac{d}{2}+1}v)\|_{L^2}2^{k\sigma}\|S_k\eta\|_{L^2}\| \eta\|_{L^2}
  \\&\lesssim\|\Lambda^{\sigma+\frac{d}{2}+1}v\|_{L^2}\|\eta\|_{\dot H^\sigma(\mathbb{T}^d)}^2.
\end{align*}

For the low frequencies $k\le k_0$ ($k_0$ being a suitable fixed positive integer), 
we estimate using~\eqref{eq:Sk} and omitting the $k_0$ dependence
\begin{align*}
 \frac{d}{dt}\sum_{0\le k\le k_0}2^{2\sigma k}\|S_k\eta\|_{L^2(\mathbb{T}^d)}^2&\lesssim
 \sum_{0\le k\le k_0}\|[v,S_k]\eta\|_{L^2}\|S_k\eta\|_{L^2}
 \\&\lesssim\|\hat{v}(\alpha)\alpha\|_{l^1_\alpha}\|\eta(t)\|_{\dot H^\sigma(\mathbb{T}^d)}^2.
\end{align*}
In the second step, we used Lemma~\ref{commutator_gain_slow_varying} (mainly in order to illustrate that the estimate is independent of $\hat v(0)$).

We now recall~\eqref{eq:Sk} and combine our estimates for high and low frequencies to conclude
\begin{align}\label{eq:Gronw}
\frac{d}{dt}\|\eta(t)\|_{\dot H^\sigma(\mathbb{T}^d)}^2&\lesssim\left(\|\Lambda^{\sigma+\frac{d}{2}+1}v\|_{L^2}+\|\hat{v}(\alpha)\alpha\|_{l^1_\alpha}\right)\|\eta(t)\|_{\dot H^\sigma(\mathbb{T}^d)}^2.
\end{align}
Finally note that since $\sigma>0$ 
\begin{align*}
 \sum_{\alpha\neq0}|\hat v(\alpha)\alpha|\le\left(\sum_{\alpha\neq0}|\hat v(\alpha)|^2|\alpha|^{2(1+\frac{d}{2}+\sigma)}\right)^\frac{1}{2}\left(\sum_{\alpha\neq0}|\alpha|^{-d-2\sigma}\right)^\frac{1}{2}\lesssim\|\Lambda^{\sigma+\frac{d}{2}+1}v\|_{L^2}.
\end{align*}
Hence, Gronwall's inequality applied to~\eqref{eq:Gronw} yields the claim. 
\end{proof}
\begin{remark*}
 The reader interested in optimising the regularity required for $v$ may consult the comprehensive analysis in~\cite[Theorem~3.14~\&~Lemma~2.100]{bahouri_fourier_2011}.
\end{remark*}

\section{Examples of \texorpdfstring{$\gamma$}{}-RE flows}\label{app:RE-examples}
In this section we provide examples which show that in general the classes of $\gamma$-relaxation enhancing flows introduced in Definition~\ref{def:gammaRE} are different for different $\gamma$. Our construction is an adaptation of~\cite[Proposition~6.2]{constantin_diffusion_2008}. 
\begin{proposition}
 For any $\gamma>\frac{1}{2}$ and any (small) $\varepsilon>0$ there exists a smooth, divergence-free vector field $u(x)$ on $\mathbb{T}^2$ such that the induced unitary evolution $U$ on $L^2(\mathbb{T}^2)$ has discrete spectrum and all non-constant eigenfunctions lie in $H^{\gamma-\varepsilon}\setminus H^{\gamma+\varepsilon}$. In particular, $u$ is $2(\gamma+\varepsilon)$-RE but not $2(\gamma-\varepsilon)$-RE.
\end{proposition}

\begin{proof}[Sketch proof]
 The proof adapts the construction in~\cite[Proposition 6.2]{constantin_diffusion_2008}.
We therefore only point out the necessary modifications. Let $\alpha\in\mathbb{R}\setminus\mathbb{Q}$ be a positive Liouvillean number. Then, by~\cite[Proposition 6.3]{constantin_diffusion_2008} (see also the original statement in~\cite[Theorem~4.5]{katok2001cocycles}) there exists
 a smooth function $h\in C^\infty(\T^1)$ and a nowhere continuous, integrable function $\tilde R$ on $\T^1$ such that 
 \begin{align}\label{eq:homo}
  \tilde R(\xi+\alpha)-\tilde R(\xi)=h(\xi) \text{ for all }\xi\in\T^1.
 \end{align}
 Since $\tilde R$ is integrable, it can naturally be identified with an element in $H^{\sigma}(\mathbb{T}^2)$ for sufficiently small $\sigma\in\mathbb{R}$. Thus, we can define 
 \begin{align*}
  r:=\inf\{s\in\mathbb{R}: \Lambda^{-s} \tilde R\in H^\gamma\}.
 \end{align*}
The discontinuity of $\tilde R$ and $\gamma>\frac{1}{2}$ imply that $r\in(0,\infty)$.
We now set $R:=\Lambda^{-r}\tilde R$ and $Q:=\Lambda^{-r}h+1$.
Let further $\varepsilon>0$ be small enough such that $\gamma-\varepsilon>\frac{1}{2}$.
Clearly
\begin{align}\label{eq:reg}
 R\in H^{\gamma-\varepsilon}(\T^1)\setminus H^{\gamma+\varepsilon}(\T^1), 
 \end{align}
 and thanks to the Sobolev embedding into H\"older spaces we may henceforth identify $R$ with its H\"older continuous representative. 
 Furthermore
 \begin{align*}
 Q\in C^\infty(\T^1)\;\text{ with }\;\int_{\T^1} Q=1,
\end{align*}
and from~\eqref{eq:homo} we deduce 
 \begin{align}\label{eq:homoMod}
 R(\xi+\alpha)-R(\xi)=Q(\xi)-1 \text{ for all }\xi\in\T^1.
 \end{align}
 Thanks to equation~\eqref{eq:homoMod} and the smoothness of $Q$ we may now proceed as in the proof of~\cite[Proposition 6.2]{constantin_diffusion_2008}. Our arguments only deviate when it comes to determining the regularity of the eigenfunctions $\psi^w_{nl}\in L^2(\T^2)$, where we use the same notation as in \cite{constantin_diffusion_2008}. For this part, let us recall (cf.~\cite[equation~(6.2)]{constantin_diffusion_2008}) that the eigenfunctions have the form 
 \begin{align*}
  \psi(x,y):=\psi_{nl}^w(x,y)=\zeta(x,y)\text{e}^{2\pi\text{i}(n\alpha+l)R(x-\alpha y)},
 \end{align*}
where $n,l\in \mathbb{Z}$. Here $\zeta(x,y)$ is a smooth complex-valued function with $|\zeta|=1$, which is not periodic in $y$.
To complete the proof, it remains to show that the regularity of $R$ implies the asserted regularity of $\psi.$ The remaining steps are then exactly the same as in~\cite{constantin_diffusion_2008}. 

Regarding the regularity of $\psi$, we may henceforth assume $(n,l)\neq(0,0)$ since otherwise the explicit form of $\zeta$ in \cite[equation~(6.2)]{constantin_diffusion_2008} implies that $\psi$ is constant.
Since $R$ is H\"older continuous and bounded,  the regularity~\eqref{eq:reg} implies that for any $\lambda\in \mathbb{R}^*$
\begin{align}\label{eq:Rlambda-reg}
 R_\lambda(\xi):=\text{e}^{\text{i}\lambda R(\xi)}\in H^{\gamma-\varepsilon}(\T^1)\setminus H^{\gamma+\varepsilon}(\T^1).
\end{align}
This can easily be seen by noting that $\text{e}^{\text{i}\lambda\,\cdot}:\mathbb{R}\to \mathbb{S}^1$ is a local $C^\infty$ diffeomorphism and by using standard fractional chain rule/Moser type estimates (see e.g.~\cite[Chapter~3]{taylor1991pseudodifferential}). 

Let us next fix $\lambda=2\pi(n\alpha+l)$, which is different from $0$, and consider the function $$\Theta_\lambda(x,y):=R_\lambda(x-\alpha y):\T^1\times \T^1_{\alpha^{-1}}\to\mathbb{S}^1,$$ where $\T^1\times\T^1_{\alpha^{-1}}$ denotes the periodic box $[0,1)\times[0,\alpha^{-1})$. By using the explicit definition of $\|\cdot\|_{\dot H^s}$  (in terms of Fourier coefficients) one quickly finds
\begin{align*}
 \|\Theta_\lambda\|_{\dot H^s(\T^1\times\mathbb{T}^1_{\alpha^{-1}})}=C_{s,\alpha} \|R_\lambda\|_{\dot H^s(\T^1)}
\end{align*}
for some positive constant $C_{s,\alpha}>0$.
Thus, \eqref{eq:Rlambda-reg} yields
\begin{align}\label{eq:Theta-reg}
 \Theta_\lambda\in (H^{\gamma-\varepsilon}\setminus H^{\gamma+\varepsilon})(\T^1\times \T^1_{\alpha^{-1}}).
\end{align}
To conclude the regularity
\begin{align*}
 \psi\in (H^{\gamma-\varepsilon}\setminus H^{\gamma+\varepsilon})(\T^1\times \T^1)
\end{align*}
one can use a smooth partition of unity of $\mathbb{T}^2$ in $y$-direction corresponding to a finite number of overlapping cylinders of height $\frac{1}{2}\alpha^{-1}$ (if $\alpha>1$). This allows us to split $\psi$ into a finite sum of functions, which may be considered (by first (smoothly) extending by zero to $\mathbb{T}^1\times\mathbb{R}^1$ and then suitably periodising) as being defined on $\T^1\times \T^1_{\alpha^{-1}}$. Each of these summands is the product of a smooth function with $\Theta_\lambda$ so that~\eqref{eq:Theta-reg} implies $\psi\in H^{\gamma-\varepsilon}$. In order to see $\psi\not\in H^{\gamma+\varepsilon}$ one can use similar arguments together with the fact that $|\zeta|=1$ everywhere.

\end{proof}

\end{appendix}

\section*{Acknowledgements}
KH is supported by the doctoral training centre MASDOC at the University of Warwick, which is funded by the EPSRC grant EP/HO23364/1.
JLR is partially supported by the European Research Council grant 616797. 


  
   K.\;Hopf, Mathematics Institute, University of Warwick, Coventry, CV4 7AL, United Kingdom\\
  \textit{E-mail address:} \begin{tt}k.hopf@warwick.ac.uk\end{tt}\\
  
    J.\;L.\;Rodrigo, Mathematics Institute, University of Warwick, Coventry, CV4 7AL, United Kingdom\\
  \textit{E-mail address:} \begin{tt}j.l.rodrigo@warwick.ac.uk\end{tt}

\end{document}